\documentclass{article} 

\usepackage{amsmath,amsthm,amssymb,amsfonts}

\usepackage{geometry} 
\newtheorem{Theorem}{Theorem}[section]
\theoremstyle{plain}

\newtheorem{Lemma}[Theorem]{Lemma}
\newtheorem{Definition}[Theorem]{Definition}
\newtheorem{Remark}[Theorem]{Remark}
\newtheorem{Corollary}[Theorem]{Corollary}
\newtheorem{Proposition}[Theorem]{Proposition}

\numberwithin{equation}{section}

\newcommand{\R}{\mathbb R}

\title{Branching Brownian Motion \\ with catalytic branching at the origin}
\author{Sergey Bocharov and Simon C. Harris\\University of Bath}

\begin{document}
\maketitle

\begin{abstract}
We consider a branching Brownian motion 
in which binary fission takes place only when particles are at the origin at a rate $ \beta>0$ on the local time scale.
We obtain results regarding the asymptotic behaviour of the number 
of particles above $ \lambda t$ at time $t$, for $\lambda>0$. 
As a corollary, we establish the almost sure asymptotic speed of the rightmost particle. 
We also prove a Strong Law of Large Numbers for this catalytic branching Brownian motion.
\end{abstract}
\section{Introduction}
\subsection{Model}
In this article we study a branching Brownian motion in which binary fission takes 
place at the origin at rate $ \beta > 0$ on the local time scale. 
That is, if $(X_t: t\leq\tau)$ is the path and $(L_s :s\leq \tau)$ is the local time at the origin of the initial Brownian particle particle up until the first fission time $\tau$,
then the first birth occurs {at the origin} as soon as an independent exponential amount of local time has been accumulated with $L_\tau \stackrel{d}{=} Exp(\beta)$ and $X_\tau=0$.
Once born, particles move off independently from their birth position (at the origin), replicating the behaviour of the parent, and so on.
Heuristically, we have an inhomogeneous branching Brownian motion with instantaneous branching rate
 $\beta(x):= \beta\, \delta_0(x)$ , since we can informally think of Brownian local time at the origin as $L_t = \int_0^t \delta_0(X_s) \,\textrm{d}s $,
where $\delta_0$ is the unit Dirac-mass at $0$. 

Although BBM models have been very widely studied, the degenerate nature of such catalytic branching at the origin means that the above BBM model needs some special treatment.
Related models with catalytic branching have been extensively studied in the context of superprocesses; for example, see Dawson \& Fleischmann \cite{34}, Fleischmann \& Le Gall \cite{35} 
or Engl\"ander \& Turaev \cite{19}. In the discrete setting, catalytic branching random walk models have recently been considered by, for example, Carmona \& Hu \cite{28} and D\"oring \& Roberts \cite{33}.

\subsection{Main results} 
In this section, after first setting up some notation, we will state our main results for BBM with catalytic branching at the origin (presenting them in the order that we will prove them). 

We denote the set of particles present in the system at time $t$ by $N_t$, labelling particles according to the usual Ulam-Harris convention. 
If $u \in N_t$ then the position of particle $u$ at time $t$ is $X^u_t$ and its historical path up to time $t$ is $(X^u_s)_{0 \leq s \leq t}$. Also, we denote the local 
time process of a particle $u \in N_t$ by $(L^u_s)_{0 \leq s \leq t}$. The law of the branching process started with a single initial particle at $x$ is 
denoted by $P^x$ with the corresponding expectation $E^x$.

Firstly, we shall calculate the \emph{expected} population growth.
\begin{Lemma}[Expected total population growth]
\label{E_N} For $t>0$, 
\[
E \Big( |N_t| \Big) = 2 \Phi(\beta \sqrt{t}) e^{\frac{\beta^2}{2}t} \sim 2e^{\frac{\beta^2}{2}t} \text{ as } t \to \infty \text{.}
\]
\end{Lemma}
\begin{Lemma}[Expected population growth rates]
For $ \lambda > 0$, let $ N_t^{ \lambda t} := \{ u \in N_t : X_t^u > \lambda t \} $ be the set of particles 
that have an average velocity greater than $ \lambda$ at time $t$. Then, as $t \to \infty$,
\label{E_Nlambda}
\begin{equation*}
\dfrac{1}{t} \log E \Big( |N_t^{\lambda t}| \Big) \to 
\Delta_{\lambda} := \left\{
\begin{array}{rl}
\frac{1}{2} \beta^2 - \beta \lambda & \text{if } \lambda < \beta \\
- \frac{1}{2} \lambda^2 & \text{if } \lambda \geq \beta 
\end{array} \right.
\end{equation*}
\end{Lemma}
Note, the expected growth rate of particles with velocities greater than $\lambda>0$, $\Delta_{\lambda}$, is positive or negative according to whether $ \lambda$ is less than or greater than $\beta/2$, respectively. 
That is, the \emph{expectation} speed of the rightmost particle is $\beta/2$. (Also note, by symmetry, similar results hold throughout for particles with negative velocities.)

Next, we consider the \emph{almost sure} asymptotic behaviour of the population. 
\begin{Theorem}[Almost sure total population growth rate]
\label{N_as}
\[
\lim_{t \to \infty} \frac{ \log |N_t|}{t} = \frac{1}{2} \beta^2 \qquad P \text{-a.s.}
\]
\end{Theorem}
\begin{Theorem}[Almost sure population growth rates]
\label{Nlambda_as}
Let $ \lambda > 0$, then:
\begin{enumerate}
\item if $\lambda > \frac{ \beta}{2}$ then $ \lim_{t \to \infty} |N_t^{ \lambda t}| = 0 \ P$- a.s.
\item if $\lambda < \frac{ \beta}{2}$ then $ \lim_{t \to \infty} \dfrac{\log |N_t^{ \lambda t}|}{t} 
= \Delta_{ \lambda} = \frac{1}{2} \beta^2 - \beta \lambda \ P$-a.s.
\end{enumerate}
\end{Theorem}
From Theorem \ref{Nlambda_as}, we immediately recover the speed of the rightmost particle,
\[
R_t := \sup_{u \in N_t} X^u_t \ \text{, } \qquad t \geq 0 \text{.}
\]
\begin{Corollary}[Rightmost particle speed]
\label{Th_rightmost}
\[
\lim_{t \to \infty} \frac{R_t}{t} = \dfrac{\beta}{2} \quad P \text{-a.s.}
\]
\end{Corollary}
We can also say something about the rare events of $|N_t^{\lambda t}|$ being positive when we typically do not find particles with speeds $\lambda>\frac{\beta}{2}$.
\begin{Lemma}[Unusually fast particles]
\label{log_P}
For $ \lambda > \frac{ \beta}{2}$, 
\[
\lim_{t \to \infty} \frac{\log P (|N_t^{\lambda t}| \geq 1)}{t} = \Delta_{ \lambda} = - \frac{1}{2} \lambda^2 \text{.}
\]
\end{Lemma}

\newpage
Finally, our main theorem gives a strong law of large numbers for the catalytic BBM:
\begin{Theorem}[SLLN]
\label{new_SLLN}
Let $f : \mathbb{R} \to \mathbb{R} $ be some Borel-measurable bounded function. Then
\[
\lim_{t \to \infty} e^{- \frac{ \beta^2}{2}t} \sum_{u \in N_t} f(X^u_t) = M_{\infty} \int_\R f(x) \beta e^{- \beta |x|} \mathrm{d}x  \quad P \text{-a.s.,}
\]
where $M_{\infty}$ is the almost sure limit of the $P$-uniformly integrable additive martingale 
\[
M_t = \sum_{u \in N_t} \exp \big\{ - \beta |X^u_t| - \frac{1}{2} \beta^2 t \big\} \text{.}
\]
(Note: the martingale $(M_t)_{t \geq 0}$ will be discussed in detail in section \ref{spinesection}.) 
\end{Theorem}
 
One can observe that taking $f(\cdot) \equiv 1$ in Theorem \ref{new_SLLN} would give Lemma \ref{E_N} 
and an even stronger result than given in Theorem \ref{N_as}. 
However, our proof of Theorem \ref{new_SLLN} relies on Theorem  \ref{N_as}, which in turn relies on Lemma \ref{E_N}. 
Thus, it will be necessary that we prove results in the order presented above.

The rest of this article is arranged as follows. In Section 2 we recall some basic facts regarding the local times. We also introduce a Radon-Nikodym derivative 
that puts a drift towards the origin onto a Brownian motion. This will be useful in the subsequent analysis of the model.
In Section 3, we recall some standard techniques for branching processes including spines and additive martingales.
Section 4 is devoted to the proofs of Lemmas \ref{E_N} and \ref{E_Nlambda}.
We will prove Theorem \ref{N_as} in Section 5, making use of the additive martingale $(M_t)_{t \geq 0}$ mentioned above. 
Section 6 contains the proofs of Theorem \ref{Nlambda_as}, Corollary \ref{Th_rightmost} and Lemma \ref{log_P}. 
Finally, in Section 7 we give the proof of Theorem \ref{new_SLLN}, this being largely based on extending the results found in Engl\"ander, Harris \& Kyprianou
\cite{17}.
\section{Single-particle results}
Basic information about local times and the excursion theory can be found in many 
textbooks on Brownian motion (for example, see \cite{32}). Also a good introduction is given in 
the paper of C. Rogers \cite{20}. Let us recall a few basic facts.

Suppose $(X_t)_{t \geq 0}$ is a standard Brownian motion on some probability space under probability measure $ \mathbb{P}$. 
Let $(L_t)_{t \geq 0}$ be its local time at $0$. Then $(L_t)_{t \geq 0}$ satisfies
\[
L_t = \lim_{ \epsilon \to 0} \frac{1}{2 \epsilon} \int_0^t \mathbf{1}_{ \{ X_s \in (- \epsilon , \epsilon ) \} } \mathrm{d}s 
\]
for every $t \geq 0$.
The next famous result is Tanaka's formula:
\[
|X_t| = \int_0^t sgn(X_s) \mathrm{d}X_s + L_t \text{,}
\]
where
\begin{equation*}
sgn(x) = 
\left\{ \begin{array}{rl}
1 & \text{ if } x > 0 \\
-1 & \text{ if } x \leq 0
\end{array} \right.
\end{equation*}
In a non-rigorous way this can be thought of as It\^o's formula applied to $f(x) = |x|$, 
where $f'(x) = sgn(x)$, $f''(x) = 2 \delta_0(x)$ (where $ \delta_0$ is the Dirac 
delta function). Then one can think of $L_t$ as $ \int_0^t \delta_0(X_s) \mathrm{d}s$.

Another useful result is the following theorem due to L\'evy.
\begin{Theorem}[L\'evy]
\label{Levy_Theorem}
Let $(S_t)_{t \geq 0}$ be the running supremum of $X$. That is, $S_t = \sup_{0 \leq s \leq t} X_s$. Then 
\[
(S_t, S_t - X_t)_{t \geq 0} \stackrel{d}{=} (L_t, |X_t|)_{t \geq 0}
\] 
\end{Theorem}
From Theroem \ref{Levy_Theorem} and the Reflection Principle it follows that $ \forall t \geq 0$ 
\[
L_t \stackrel{d}{=} S_t \stackrel{d}{=} |X_t| \stackrel{d}{=} |N(0,t)| \text{.}
\]
It also follows that $(Z_t)_{t \geq 0} := (|X_t| - L_t)_{t \geq 0} = (\int_0^t sgn(X_s) \mathrm{d}X_s)_{t \geq 0}$ is a 
standard Brownian motion under $\mathbb{P}$, hence for any $ \gamma \in \mathbb{R}$
\[
\exp \Big\{ \gamma \big( |X_t| - L_t \big) - \dfrac{1}{2} \gamma^2 t \Big\} 
= \exp \Big\{ \gamma Z_t - \dfrac{1}{2} \gamma^2 t \Big\} \text{ ,} \qquad t \geq 0
\]
is a martingale. And more generally, for $ \gamma (\cdot)$ a smooth path we have the Girsanov 
martingale 
\begin{align}
\label{M_tilde_0}
W_t = &\exp \Big\{ \int_0^t \gamma(s) \mathrm{d} Z_s - \dfrac{1}{2} \int_0^t 
\gamma^2(s) \mathrm{d}s \Big\} \nonumber \\
\stackrel{\text{Tanaka}}{=} \ &\exp \Big\{ \int_0^t \gamma(s) sgn(X_s) \mathrm{d}X_s 
- \dfrac{1}{2} \int_0^t \gamma^2(s) \mathrm{d}s \Big\}
\end{align}
Used as the Radon-Nikodym derivative it puts the instantaneous drift $ sgn(X_t) \gamma(t)$ on the process $(X_t)_{t \geq 0}$. 
Let us restrict ourselves to the case $ \gamma( \cdot) \equiv - \gamma < 0$ so that $W$ puts the 
constant drift $ \gamma$ towards the origin on $(X_t)_{t \geq 0}$. The following result is taken from \cite{15}.
\begin{Proposition} 
\label{prop_drift_to_0}
Let $ \mathbb{Q}$ be the probability measure defined as
\[
\dfrac{\mathrm{d} \mathbb{Q}}{\mathrm{d} \mathbb{P}} \Big\vert_{ \hat{\mathcal{F}}_t} = 
\exp \Big\{ - \gamma \big( |X_t| - L_t \big) - \dfrac{1}{2} \gamma^2 t \Big\} \text{ ,} \qquad t \geq 0 \text{,}
\]
where $(\hat{\mathcal{F}}_t)_{t \geq 0}$ is the natural filtration of $(X_t)_{t \geq 0}$. 
Then under $ \mathbb{Q}$, $(X_t)_{t \geq 0}$ has the transition density
\[
p(t;x,y) = \dfrac{1}{2 \sqrt{2 \pi t}} \exp \Big( \gamma(|x|+|y|) - \frac{\gamma^2}{2}t 
- \frac{(x-y)^2}{2t} \Big) + \dfrac{\gamma}{4} \text{Erfc} \Big( \dfrac{|x|+|y|- \gamma t}{\sqrt{2t}}
\Big)
\]
with respect to the speed measure 
\[
m(\mathrm{d}y) = 2e^{-2 \gamma |y|} \mathrm{d}y \text{,}
\]
so that
\begin{equation}
\label{q_new}
\mathbb{Q}^x \big( X_t \in A \big) = \int_A p(t;x,y) m(\mathrm{d}y) \text{.}
\end{equation}
Here Erfc$(x) = \frac{2}{\sqrt{ \pi}} \int_x^{\infty} e^{-u^2} \mathrm{d}u
\sim \frac{1}{x \sqrt{ \pi}}e^{-x^2}$ as $x \to \infty$.

It also has the stationary probability measure
\begin{equation}
\label{pi_new}
\pi( \mathrm{d}x) = \gamma e^{- 2 \gamma |x|} \mathrm{d}x \text{.}
\end{equation}
\end{Proposition}
\section{Spines and additive martingales}
\label{spinesection}
\subsection{Spine setup}
In this section we give a brief overview of some main spine tools. The major 
reference for this section is the work of Hardy and Harris \cite{2} where all the proofs 
and further references can be found.

We let $(\mathcal{F}_t)_{t \geq 0}$ denote the natural filtration of our branching process as described in 
the introduction. We define $\mathcal{F}_\infty := \sigma(\cup_{t \geq 0} \mathcal{F}_t)$ as usual.

Let us now extend our model by identifying an infinite line of descent which we refer to as the spine and 
which is chosen uniformly from all the possible lines of descent. It is defined in the following way. 
The initial particle of the branching process begins the spine. When it splits into two new particles, one of them 
is chosen with probability $\frac{1}{2}$ to continue the spine. This goes on in the obvious way: whenever the particle 
currently in the spine splits, one of its children is chosen uniformly at random to continue the spine.

The spine is denoted by $\xi = \{ \varnothing, \xi_1, \xi_2, \cdots\}$, where $\varnothing$ is the initial particle (both in the spine 
and in the entire branching process) and $\xi_n$ is the particle in the $(n+1)^{st}$ generation of the spine. Furthermore, at time $t \geq 0$ 
we define: 
\begin{itemize}
\item $node_t(\xi) := u \in N_t \cap \xi$ (such $u$ is necessarily unique). That is, $node_t(\xi)$ is the particle in the spine alive at time $t$.
\item $n_t := |node_t(\xi)|$. Thus $n_t$ is the number of fissions that have occured along the spine by time $t$. 
\item $\xi_t := X^u_t$ for $u \in N_t \cap \xi$. So $(\xi_t)_{t \geq 0}$ is the path of the spine.
\end{itemize}
The next important step is to define a number of filtrations of our sample 
space, which contain different information about the process.
\begin{Definition}[Filtrations]$ $
\begin{itemize}
\item $ \mathcal{F}_t$ was defined earlier. It is the filtration which 
knows everything about the particles' motion and their genealogy, but it 
knows nothing about the spine.
\item We also define 
$ \tilde{\mathcal{F}}_t := \sigma \big( \mathcal{F}_t, node_t(\xi) \big) $. 
Thus $ \tilde{\mathcal{F}}$ has all the information about the branching process 
and all the information about the spine. This will be the largest 
filtration.
\item $ \mathcal{G}_t := \sigma \big( \xi_s : 0 \leq s \leq t \big)$. 
This filtration only has information about the path of the spine process, 
but it can't tell which particle $u \in N_t$ is the spine particle at time $t$.
\item $ \tilde{\mathcal{G}}_t := \sigma \big( \mathcal{G}_t , \ ( node_s(\xi) : 
0 \leq s \leq t) \big)$. This filtration knows everything about the spine 
including which particles make up the spine, but it 
doesn't know what is happening off the spine.
\end{itemize}
\end{Definition}
Note that $ \mathcal{G}_t \subset \tilde{\mathcal{G}}_t 
\subset \tilde{\mathcal{F}}_t$ and $\mathcal{F}_t \subset \tilde{\mathcal{F}}_t$. We shall use these 
filtrations to take various conditional expectations.

We let $\tilde{P}$ be the probability measure under which the branching process is defined together with the 
spine. Hence $P = \tilde{P} \vert_{\mathcal{F}_\infty}$. We shall write $\tilde{E}$ for 
the expectation with respect to $ \tilde{P}$.

 Under $\tilde{P}$ the entire branching process (with the spine) can be described in the following way.
\begin{itemize}
\item the initial particle (the spine) moves like a Brownian motion.
\item At instantaneous rate $ \beta \delta_0(\cdot)$ it splits into two new particles. 
\item One of these particles (chosen uniformly at random) continues the spine.
That is, it continues moving as a Brownian motion and branching at rate 
$ \beta \delta_0(\cdot)$. 
\item The other particle initiates a new independent $P$-branching processes 
from the position of the split.
\end{itemize}
It is not hard to see that under $\tilde{P}$ the spine's path $(\xi_t)_{t \geq 0}$ is itself a 
Brownian motion. We denote by $(\tilde{L}_t)_{t \geq 0}$ its local time at $0$.

Also, conditional on the path of the spine, $(n_t)_{t \geq 0}$ is a 
time-inhomogeneous Poisson process (or a Cox process) with instantaneous jump rate $ \beta \delta_0( \xi_t)$. 
That is, conditional on $ \mathcal{G}_t$, $k$ splits take place along the spine by time $t$
with probability 
\[
\tilde{P}(n_t = k | \mathcal{G}_t) = 
\frac{(\beta \tilde{L}_t)^k}{k!} e^{- \beta \tilde{L}_t }
\text{.}
\]
The next result (for example, see \cite{2}) is very useful in computing expectations of various quantities
\begin{Theorem}[Many-to-One Theorem]
\label{manytoone_0}
Let $f(t) \in m \mathcal{G}_t$. In other words, $f(t)$ is 
$\mathcal{G}_t$-measurable. Suppose it has the representation 
\[
f(t) = \sum_{u \in N_t} f_u(t) \mathbf{1}_{ \{ node_t( \xi) = u \} } \text{,} 
\]
where $f_u(t) \in m \mathcal{F}_t$, then
\[
E \Big( \sum_{u \in N_t} f_u(t) \Big) = \tilde{E} \Big( f(t) 
e^{\beta \tilde{L}_t} \Big) \text{.}
\]
\end{Theorem}
\subsection{Martingales}
Since $(\xi_t)_{t \geq 0}$ is a standard Brownian motion we can define the following $ \tilde{P}$-martingale 
with respect to the filtration $(\mathcal{G}_t)_{t \geq 0}$ 
using Proposition \ref{prop_drift_to_0}:
\begin{equation}
\label{M_beta}
\tilde{M}^{\beta}_t := e^{ - \beta |\xi_t| + \beta \tilde{L}_t - \frac{1}{2} \beta^2 t} \text{ ,} \qquad t \geq 0 \text{.}
\end{equation}
We also define the corresponding probability measure $ \tilde{Q}_{\beta}$ as 
\begin{equation}
\label{Q_beta0}
\dfrac{\mathrm{d} \tilde{Q}_{\beta}}{\mathrm{d} \tilde{P}} \Big\vert_{ \mathcal{G}_t} = \tilde{M}^{\beta}_t \text{ ,} \qquad t \geq 0 \text{.}
\end{equation}
Then under $ \tilde{Q}_{\beta}$, $(\xi_t)_{t \geq 0}$ has drift $ \beta$ towards the origin and from Proposition 
\ref{prop_drift_to_0} we know its exact transition density as well as its stationary distribution. 

Let us also define the martingale
\[
\tilde{M}_t := 2^{n_t}e^{- \beta \tilde{L}_t} \tilde{M}_t^{\beta} \qquad \text{, } t \geq 0 \text{,}
\]
which is the product of two $ \tilde{P}$-martingales, the first of which 
doubles the branching rate along the spine, and the second puts the drift of magnitude $\beta$ towards the origin. 
If we define a probability measure $ \tilde{Q}$ as
\begin{equation}
\label{Q_tilde1_}
\dfrac{\mathrm{d} \tilde{Q}}{\mathrm{d} \tilde{P}} \bigg\vert_{\tilde{\mathcal{F}}_t} = \tilde{M}_t \ \text{, } \qquad t \geq 0
\end{equation}
then under $\tilde{Q}$ the branching process has the following description: 
\begin{itemize}
\item The initial particle (the spine) moves like a Brownian motion with drift $\beta$ towards the origin. 
\item When it is at position $x$ it splits into two new particles at instantaneous rate $ 2 \beta \delta_0(x)$.
\item One of these particles (chosen uniformly at random) continues the spine. I.e. it 
continues moving as a biased random walk and branching at rate $ 2 \beta \delta_0(x)$.
\item The other particle initiates an unbiased branching process (as under $P$) from the position of the split.
\end{itemize}
Note that although \eqref{Q_tilde1_} only defines $ \tilde{Q}$ on events in 
$ \cup_{t \geq 0} \tilde{\mathcal{F}}_t$, Carath\'eodory's extension theorem tells that $ \tilde{Q}$ has a 
unique extension on $ \tilde{\mathcal{F}}_\infty := \sigma( \cup_{t \geq 0} \tilde{\mathcal{F}}_t)$ 
and thus \eqref{Q_tilde1_} implicitly defines $ \tilde{Q}$ on $ \tilde{\mathcal{F}}_\infty$.
We then define $Q := \tilde{Q}|_{ \mathcal{F}_{ \infty}}$ so that 
\begin{align}
\label{addtive_Mbeta}
\frac{\mathrm{d} Q}{\mathrm{d} P} \bigg\vert_{ \mathcal{F}_t} = M_t &:= \sum_{u \in N_t} \exp \Big\{ 
\big( - \beta |X_t^u| + L_t^u - \dfrac{1}{2} \beta^2 t \big) - \beta L_t^u \Big\} \nonumber \\
&=  \sum_{u \in N_t} \exp \Big\{ - \beta |X_t^u| - \dfrac{1}{2} \beta^2 t \Big\} \text{ ,}
\qquad t \geq 0 \text{.}
\end{align}
$(M_t)_{t \geq 0}$ will be referred to as the additive martingale.
\subsection{Convergence properties of $(M_t)_{t \geq 0}$}
The following theorem is a standard result for additive martingales in the study of branching processes.
\begin{Theorem}
\label{UI_Mbeta}
$(M_t)_{t \geq 0} $ is $P$-uniformly integrable and $M_{\infty} > 0 \ P$-almost surely.
\end{Theorem}
\begin{proof}
Recall the following measure-theoretic result, which gives Lebesgue's decomposition 
of $Q$ into absolutely-continuous and singular parts. It can for example be found 
in the book of R. Durrett \cite{6} (Section 4.3).
\begin{Lemma}
\label{durrett_lem2_}
For events $A \in \mathcal{F}_\infty$
\[
Q(A) = \int_A \limsup_{t \to \infty} M_t \ \mathrm{d}P + Q \Big( A \cap \big\{ \limsup_{t \to \infty} M_t = \infty \big\} \Big) \text{.}
\]
\end{Lemma}
Also a standard zero-one law, which can be found, for example,  in \cite{3} (see Lemma 3 and the proof of Theorem 2 that follows it) 
tells that $P(M_{\infty} > 0) \in \{ 0, 1 \}$.
Thus to prove Theorem \ref{UI_Mbeta} it is sufficient to show that 
\begin{equation}
\label{sllaaasd}
\limsup_{t \to \infty} M_t < \infty \quad Q \text{-a.s.}
\end{equation}
Let us consider the spine decomposition of $M_t$, another useful technique which can be found in \cite{2}:
\[
E^{\tilde{Q}} \Big( M_t \big\vert \tilde{ \mathcal{G}}_{ \infty} \Big) = \exp \Big\{ - \beta | \xi_t | - \frac{1}{2} \beta^2 t \Big\} + 
\sum_{u < node_t(\xi)} \exp \Big\{ - \beta | \xi_{S_u} | - \frac{1}{2} \beta^2 S_u \Big\} \text{,}
\]
where we refer to the first term as $spine(t)$ and the second term as $sum(t)$.

Recall that under $ \tilde{Q} $, $(\xi_t)_{t \geq 0}$ is a Brownian Motion with drift 
$ \beta$ towards the origin and $(|\xi_t| - \tilde{L}_t)_{t \geq 0}$ is a Brownian motion with drift $- \beta$. 
Thus $ t^{-1} \xi_t \to 0$ and $t^{-1} \tilde{L}_t \to \beta \ \tilde{Q}$-a.s. Also $spine(t) \leq 1$ and 
\begin{equation}
\label{kklllkk}
sum(t) \leq \sum_{u < node_t(\xi)} e^{- \frac{1}{2} \beta^2 S_u} \leq \sum_{n = 1}^{\infty} e^{- \frac{1}{2} \beta^2 S_n} \text{,}
\end{equation}
where $S_n$ is the $n^{th}$ birth on the spine. 
The birth process along the spine $(n_t)_{t \geq 0} $ conditional on the path of the spine is a time-inhomogeneous Poisson process 
(or a Cox process) with cummulative jump rate $2 \beta \tilde{L}_t$. Hence, $\tilde{Q}$-almost surely, 
$n_t \sim 2 \beta \tilde{L}_t \sim 2 \beta^2 t$, and so $S_n \sim (2 \beta^2)^{-1} n$. 

Thus there exists some  $\tilde{Q}$-a.s. finite random variable $C > 0$ such that $S_n \geq C n$ for all $n$. Substituting this into \eqref{kklllkk} we get 
\[
sum(t) \leq \sum_{n=1}^{\infty} e^{-\frac{1}{2} \beta^2 C n} \text{.}
\]
Therefore $sum(t)$ is bounded by some $ \tilde{Q}$-a.s. finite random variable. We deduce that 
\[
\limsup_{t \to \infty} E^{\tilde{Q}} \Big( M_t \big\vert \tilde{ \mathcal{G}}_{ \infty} \Big) 
= \limsup_{t \to \infty} \Big( spine(t) + sum(t) \Big) < \infty \quad \tilde{Q} \text{-a.s.}  
\]
So by Fatou's lemma, $\tilde{Q}$-almost surely, 
\begin{align*}
E^{\tilde{Q}} \Big( \liminf_{t \to \infty} M_t \big\vert \tilde{\mathcal{G}}_\infty \Big) \leq
\liminf_{t \to \infty} E^{\tilde{Q}} \Big( M_t \big\vert \tilde{\mathcal{G}}_\infty \Big) 
\leq \limsup_{t \to \infty} E^{\tilde{Q}} \Big( M_t \big\vert \tilde{\mathcal{G}}_\infty \Big)
< \infty \text{.}
\end{align*}
Then $\liminf_{t \to \infty} M_t < \infty \ \tilde{Q}$-a.s. and hence also $Q$-a.s.
Since $1/M_t$ is a positive $Q$-supermartingale, it must converge $Q$-a.s., hence
\[
\limsup_{t \to \infty} M_t = \liminf_{t \to \infty} M_t < \infty \qquad Q \text{-a.s.} 
\]
completing the proof of the theorem.
\end{proof}
The next theorem is essential in the proof of the Strong Law of Large Numbers in the last section. 
\begin{Theorem}
\label{Lp}
For $p \in (1, 2), \ (M_t)_{t \geq 0}$ is $L^p$-convergent. 
\end{Theorem}
\begin{proof}
We use similar proof as found in \cite{2}. It is sufficient to show that $E( M_t^p)$ is bounded in $t$. 
\begin{align*}
E \Big( M_t^p \Big) &= E \Big( M_t^{p - 1} 
M_t \Big) = E^Q \Big( M_t^{p - 1} \Big) = E^{ \tilde{Q}} \Big( M_t^{p - 1} \Big) \\ 
&= E^{ \tilde{Q}} \Big( E^{ \tilde{Q}} \Big( M_t^{p - 1} | \tilde{ \mathcal{G}}_{ \infty} \Big) \Big) 
\leq E^{ \tilde{Q}} \Big( \Big( E^{ \tilde{Q}} \big( M_t | \tilde{ \mathcal{G}}_{ \infty} \big) \Big)^{p-1} \Big)
\end{align*}
by Jensen's inequality. Since for $a, b \geq 0$ and $q \in (0, 1)$, $(a + b)^q \leq a^q + b^q$, 
we see that 
\begin{align*}
\Big( E^{ \tilde{Q}} \big( M_t | \tilde{ \mathcal{G}}_{ \infty} \big) \Big)^{p-1} &= \big( spine(t) + sum(t)  \big)^{p-1} \\
&\leq e^{- \frac{\beta^2}{2}(p - 1)t - \beta (p - 1) | \xi_t|} + 
\sum_{u < node_t(\xi)} e^{- \frac{\beta^2}{2}(p - 1)S_u - \beta (p - 1) | \xi_{S_u}|} 
\end{align*}
And hence
\[
E \Big( M_t^p \Big) \leq E^{ \tilde{Q}} \Big( e^{- \frac{\beta^2}{2}(p - 1)t - \beta (p - 1) | \xi_t|} \Big) + 
E^{ \tilde{Q}} \Big( \sum_{u < node_t(\xi)} e^{- \frac{\beta^2}{2}(p - 1)S_u - \beta (p - 1) | \xi_{S_u}|} \Big)
\]
The first expectation is bounded by $1$. The second one satisfies for $\epsilon > 0$ small enough 
\begin{align*}
E^{ \tilde{Q}} \Big( \sum_{u < node_t(\xi)} e^{- \frac{\beta^2}{2}(p - 1)S_u - \beta (p - 1) | \xi_{S_u}|} \Big) 
&= E^{ \tilde{Q}} \Big( \int_0^t 2 \beta \tilde{L}_s e^{- \frac{\beta^2}{2}(p - 1)s - \beta (p - 1) | \xi_{s}|} \mathrm{d}s \Big) \\ 
&= \int_0^t E^{ \tilde{Q}} \Big( 2 \beta \tilde{L}_s e^{- \frac{\beta^2}{2}(p - 1)s - \beta (p - 1) | \xi_{s}|} \Big) \mathrm{d}s \\ 
&\leq \int_0^t E^{ \tilde{Q}} \Big( \frac{2}{\epsilon} e^{\beta \epsilon \tilde{L}_s - \frac{\beta^2}{2}(p - 1)s - \beta (p - 1) | \xi_{s}|} \Big) \mathrm{d}s \\ 
&\leq \int_0^t E^{ \tilde{Q}} \Big( \frac{2}{\epsilon} e^{-\beta \tilde{L}_s + \beta |\xi_s| + \frac{\beta^2}{2}s}
e^{(1 + \epsilon) \beta ( \tilde{L}_s - |\xi_s| ) - \frac{\beta^2}{2} ps} \Big) \mathrm{d}s \\ 
&= \frac{2}{\epsilon} \int_0^t \tilde{E} \Big( e^{(1 + \epsilon) \beta ( \tilde{L}_s - |\xi_s| ) - \frac{\beta^2(1 + \epsilon)^2}{2} s} 
\Big) e^{\frac{\beta^2(1 + \epsilon)^2}{2} s - \frac{\beta^2}{2} ps} \mathrm{d}s \\ 
&= \frac{2}{\epsilon} \int_0^t e^{\frac{\beta^2(1 + \epsilon)^2}{2} s - \frac{\beta^2}{2} ps} \mathrm{d}s \text{,}
\end{align*}
which is bounded for $ \epsilon$ chosen sufficiently small ($ \epsilon < \sqrt{p} - 1$).
\end{proof}
\section{Expected population growth}
\subsection{Asymptotic expected growth of $|N_t|$}
We prove Lemma \ref{E_N} using the Many-to-One Theorem. 
\begin{proof}[Proof of Lemma \ref{E_N}]
From Theorem \ref{manytoone_0} we have 
\[
E \Big( |N_t| \Big) = E \Big( \sum_{u \in N_t} 1 \Big) = \tilde{E} \Big( e^{\beta \tilde{L}_t} \Big) \text{.}
\]
Using the fact that $\tilde{L}_t \stackrel{d}{=} |N(0, t)|$ it is then easy to check that 
\[
\tilde{E}( e^{\beta \tilde{L}_t}) = 2 \Phi ( \beta \sqrt{t}) e^{ \frac{\beta^2}{2}t} \text{,} 
\]
where $ \Phi(x) = \mathbb{P} (N(0, 1) \leq x)$.

We can also find a good estimate of $\tilde{E}( e^{\beta \tilde{L}_t})$ using the change of measure 
from \eqref{Q_beta0}, which is instructive for our purposes:
\[
\tilde{E} \Big( e^{\beta \tilde{L}_t} \Big) = \tilde{E} \Big( e^{\beta \tilde{L}_t - 
\beta |\xi_t| - \frac{1}{2} \beta^2 t}e^{ \beta |\xi_t| + \frac{1}{2} \beta^2 t} \Big) 
= \tilde{E} \Big( \tilde{M}_t^{\beta} e^{ \beta |\xi_t| + \frac{1}{2} \beta^2 t} \Big) = 
E^{\tilde{Q}_{\beta}} \Big( e^{ \beta |\xi_t|} \Big) e^{\frac{1}{2} \beta^2 t} \text{.}
\]
Then, using the stationary measure, from \eqref{pi_new} we have 
\[
E^{\tilde{Q}_{\beta}} \Big( e^{ \beta |\xi_t|} \Big) \to \int_{- \infty}^{\infty} e^{\beta |x|} \pi( \mathrm{d}x) 
= \int_{- \infty}^{\infty} e^{\beta |x|} \beta e^{-2 \beta |x|} \mathrm{d}x = 2 \text{.}
\]
Thus
\[
E \Big( |N_t| \Big) \sim 2 e^{\frac{\beta^2}{2}t} \text{,}
\]
\end{proof}
\subsection{Asymptotic expected behaviour of $N_t^{\lambda t}$}
Let us now prove that $ \lim_{t \to \infty} \frac{1}{t} \log E (|N_t^{\lambda t}|) = \Delta_{\lambda}$, where 
$ N_t^{ \lambda t} = \{ u \in N_t : X_t^u > \lambda t \} $ and 
\begin{equation*}
\Delta_{\lambda} = \left\{
\begin{array}{rl}
\frac{1}{2} \beta^2 - \beta \lambda & \text{if } \lambda < \beta \\
- \frac{1}{2} \lambda^2 & \text{if } \lambda \geq \beta 
\end{array} \right.
\end{equation*}
\begin{proof}[Proof of Lemma \ref{E_Nlambda}]
Following the same steps as in the proof of Lemma \ref{E_N} above we get 
\begin{align*}
E \Big( |N_t^{\lambda t}| \Big) &= E \Big( \sum_{u \in N_t} \mathbf{1}_{ \{ X_t^u > \lambda t \} } \Big) 
= \tilde{E} \Big( e^{\beta \tilde{L}_t} \mathbf{1}_{ \{ \xi_t > \lambda t \} } \Big) \\
&= E^{\tilde{Q}_{\beta}} \Big( e^{\beta |\xi_t|} \mathbf{1}_{ \{ \xi_t > \lambda t \} } \Big) e^{\frac{1}{2} \beta^2 t} \\
&= \int_{\lambda t}^{ \infty} e^{ \beta x} p(t;0,x) m( \mathrm{d}x) e^{\frac{1}{2} \beta^2 t} \\
&= \int_{\lambda t}^{\infty} e^{\beta x} \Big( \dfrac{1}{2 \sqrt{2 \pi t}} \exp \big( \beta x - 
\frac{\beta^2}{t} - \frac{x^2}{2t} \big) + \frac{\beta}{4} \text{Erfc} \big( \frac{x - \beta t}{ \sqrt{2t}} 
\big) \Big) 2e^{-2 \beta x} \mathrm{d}x \ e^{\frac{1}{2} \beta^2 t} \\
&= \Big( \int_{\lambda t}^{\infty} \dfrac{1}{\sqrt{2 \pi t}} e^{- \frac{x^2}{2t}} 
\mathrm{d}x \Big) + \Big( \frac{\beta}{2} \int_{\lambda t}^{\infty} \text{Erfc} 
\big( \frac{x - \beta t}{ \sqrt{2t}} \big) e^{- \beta x} \mathrm{d}x \Big) e^{\frac{1}{2} \beta^2 t} \text{,}
\end{align*}
where $m(\mathrm{d}x)$ and $p$ were defined in Proposition \ref{prop_drift_to_0}.

Then for some functions $\epsilon_i(t)$ satisfying $ \log \epsilon_i(t) = o(t)$ we have the 
following:
\[
\int_{\lambda t}^{\infty} \dfrac{1}{\sqrt{2 \pi t}} e^{- \frac{x^2}{2t}} \mathrm{d}x = 
\epsilon_1(t) e^{-\frac{\lambda^2}{2}t} \text{,}
\]
\begin{equation*}
\Big( \frac{\beta}{2} \int_{\lambda t}^{\infty} \text{Erfc} 
\big( \frac{x - \beta t}{ \sqrt{2t}} \big) e^{- \beta x} \mathrm{d}x \Big) e^{\frac{1}{2} \beta^2 t} = \left\{
\begin{array}{rl}
\epsilon_2(t) e^{-\frac{\lambda^2}{2}t} & \text{if } \lambda \geq \beta \\
 \epsilon_3(t) e^{- \beta \lambda t + \frac{\beta^2}{2}t} + \epsilon_4(t) e^{\frac{- \beta^2}{2}t} & \text{if } \lambda < \beta 
\end{array} \right.
\end{equation*}
where we have used that Erfc$(x) \sim \frac{1}{x \sqrt{\pi}} e^{- x^2 }$ as 
$x \to \infty$ and Erfc$(x) \to 2$ as $x \to - \infty$. Thus 
\begin{equation*}
E(|N_t^{\lambda t}|) = \left\{
\begin{array}{rl}
\epsilon_6(t) e^{-\frac{\lambda^2}{2}t} & \text{if } \lambda \geq \beta \\
\epsilon_7(t) e^{- \beta \lambda t + \frac{\beta^2}{2}t} & \text{if } \lambda < \beta 
\end{array} \right.
\end{equation*}
which proves the result after taking the logarithm and dividing by $t$.
\end{proof}
Again, we could have evaluated $\tilde{E} \Big( e^{\beta \tilde{L}_t} \mathbf{1}_{ \{ \xi_t > \lambda t \} } \Big) $ explicitly 
by using the joint density of $ \xi_t$ and $ \tilde{L}_t$ (for example, see \cite{21}):
\begin{equation}
\label{joint_l_xi}
\tilde{P} \big( \xi_t \in \mathrm{d}x , \tilde{L}_t \in \mathrm{d}y \big) = 
\frac{|x| + y}{ \sqrt{2 \pi t^3}} \exp \Big\{ - \frac{(|x| + y)^2}{2 t}
\Big\} \mathrm{d}x \mathrm{d}y \text{ ,} \qquad x \in \mathbb{R} \text{, } y > 0 \text{.}
\end{equation}
Lemmas \ref{E_N} and \ref{E_Nlambda} can also be proved via excursion 
theory (for example, see \cite{18}). The proofs that we presented here (using the change of measure) in particular suggest the importance of 
the additive martingale $(M_t)_{t \geq 0}$ in the study of the model. 
In the next section we shall see one simple application of this martingale.
\section{Almost sure asymptotic growth of $|N_t|$}
In this section we prove Theorem \ref{N_as} which says that $\log |N_t| \sim \frac{1}{2} \beta^2 t \ P$-almost surely.
\begin{proof}[Proof of Theorem \ref{N_as}]
Let us first obtain the lower bound: 
\begin{equation}
\label{N_as_lower}
\liminf_{t \to \infty} \frac{\log |N_t|}{t} \geq \frac{1}{2} \beta^2 \quad P \text{-a.s.}
\end{equation}
We observe that
\[
M_t =  \sum_{u \in N_t} \exp \Big\{ - \beta |X_t^u| - \dfrac{1}{2} \beta^2 t \Big\}
\leq |N_t| e^{- \frac{1}{2} \beta^2 t} \text{,}
\]
hence $ \log M_t \leq \log |N_t| - \frac{1}{2} \beta^2 t$ and so $t^{-1} \log |N_t| \geq \frac{1}{2} \beta^2 + t^{-1} \log M_t$. 
Using the fact that $ \lim_{t \to \infty} M_t > 0 \ P$-a.s. from Theorem \ref{UI_Mbeta}, we find that
\[
\liminf_{t \to \infty} \frac{\log |N_t|}{t} \geq \frac{1}{2} \beta^2 \text{.}
\]
Let us now establish the upper bound:
\begin{equation}
\label{N_as_upper}
\limsup_{t \to \infty} \frac{\log |N_t|}{t} \leq \frac{1}{2} \beta^2 \quad P \text{-a.s.}
\end{equation}
We first prove \eqref{N_as_upper} on integer (or other lattice) times. Take $ \epsilon > 0$. Then
\[
P \big( |N_t| e^{-(\frac{1}{2} \beta^2 + \epsilon)t} > \epsilon \big) \leq 
\dfrac{E |N_t| e^{-(\frac{1}{2} \beta^2 + \epsilon)t}}{\epsilon} \sim \dfrac{2}{ \epsilon} 
e^{- \epsilon t} 
\]
using the Markov inequality and Theorem \ref{E_N}. So 
\[
\sum_{n = 1}^{ \infty} P \big( |N_n| e^{-(\frac{1}{2} \beta^2 + \epsilon)n} > \epsilon \big) < \infty \text{.}
\]
Thus by the Borel-Cantelli lemma
\[
|N_n| e^{-(\frac{1}{2} \beta^2 + \epsilon)n} \to 0 \ P \text{-a.s. as} \ n \to \infty \text{.}
\]
Taking the logarithm we get
\[
\big(- \frac{1}{2} \beta^2 - \epsilon \big) n + \log |N_n| \to - \infty \text{.}
\]
Hence
\[
\limsup_{n \to \infty} \dfrac{\log |N_n|}{n} \leq \dfrac{1}{2} \beta^2 + \epsilon 
\]
and taking the limit $ \epsilon \to 0$ we get the desired result. To get the convergence over any real-valued 
sequence we note that $|N_t|$ is an increasing process and so
\[
\frac{\log |N_t|}{t} \leq \frac{\lceil t \rceil}{t} \frac{\log |N_{\lceil t \rceil}|}{\lceil t \rceil} \text{.}
\]
Hence
\[
\limsup_{t \to \infty} \frac{\log |N_t|}{t} \leq 
\limsup_{t \to \infty} \frac{\log |N_{\lceil t \rceil}|}{\lceil t \rceil} \leq \frac{1}{2} \beta^2 \text{.}
\]
Combining \eqref{N_as_upper} and \eqref{N_as_lower} now proves Theorem \ref{N_as}.
\end{proof}
\section{Almost sure asymptotic behaviour of $|N_t^{ \lambda t}|$}
In this section we prove Theorem \ref{Nlambda_as}. Namely, that
\[
\frac{\log |N_t^{\lambda t}|}{t} \to \Delta_{\lambda} \ P \text{-a.s. if } \lambda < \frac{\beta}{2}
\]
and
\[
|N_t^{\lambda t}| \to 0 \ P \text{-a.s. if } \lambda > \frac{\beta}{2} \text{.}
\]
We break the proof into two parts. In subsection 6.1 
we prove the upper bound and in subsection 6.2 the lower bound. Also  in subsections 6.3 and 6.4 present the 
proofs of Lemma \ref{log_P}, saying that $\lim_{t \to \infty} t^{-1} P(|N_t^{\lambda t}| \geq 1) = \Delta_{\lambda}$ if 
$ \lambda > \frac{\beta}{2}$, and Corollary \ref{Th_rightmost}, saying that $\lim_{t \to \infty} t^{-1} R_t = \frac{\beta}{2}$, 
where $R_t$ is the position of the rightmost particle at time $t$. 
\subsection{Upper bound}
\begin{Lemma}
\label{Nlambda_upper}
\[
\limsup_{t \to \infty} \dfrac{\log |N_t^{\lambda t}|}{t} \leq \Delta_{ \lambda} \quad P \text{-a.s.}
\]
\end{Lemma}
The upper bound can be proved in a similar way to the upper 
bound on $|N_t|$ (recall \ref{N_as_upper}). The main difference comes from the fact that $(|N_t^{\lambda t}|)_{t \geq 0}$ 
is not an increasing process and so getting convergence along any real time sequence requires some 
extra work.
\begin{proof}
Take $ \epsilon > 0$ and consider events
\[
A_n = \Big\{ \sum_{u \in N_{n+1}} \mathbf{1}_{\{ \sup_{s \in [n,n+1]} X^u_s \ \geq \ \lambda n\}} 
> e^{(\Delta_{\lambda} + \epsilon)n} \Big\} \text{.}
\]
If we can show that $P(A_n)$ decays to $0$ exponentially fast then by the Borel-Cantelli Lemma 
we would have $P(A_n  \ i.o. ) = 0$ and that would be sufficient to get the result.

By the Markov inequality and the Many-to-one theorem (Theorem \ref{manytoone_0}) we have
\begin{align*}
P \big( A_n \big) &\leq E \Big( \sum_{u \in N_{n+1}} \mathbf{1}_{\{ \sup_{s \in [n,n+1]} X^u_s \ \geq \ \lambda n\}} \Big) 
e^{-(\Delta_{\lambda} + \epsilon)n} \\
&= \tilde{E} \Big( e^{ \beta \tilde{L}_{n+1}} \mathbf{1}_{\{ \sup_{s \in [n,n+1]} \xi_s \ \geq \ \lambda n\}} \Big) 
e^{-(\Delta_{\lambda} + \epsilon)n} \\
&= \tilde{E} \Big( e^{ \beta \tilde{L}_{n+1}} \mathbf{1}_{\{ \xi_{n+1} + \bar{ \xi}_n \ \geq \ \lambda n\}} \Big) 
e^{-(\Delta_{\lambda} + \epsilon)n} \text{,}
\end{align*}
where $\bar{ \xi}_n := \sup_{s \in [n, n+1]} (\xi_s - \xi_{n+1})$ is a sequence of i.i.d. random variables equal in distribution to 
$\sup_{s \in [0,1]} \xi_s$ and $(\xi_t)_{t \geq 0}$ is a standard Brownian motion under $ \tilde{P}$.

To give an upper bound on the expectation we split it according to whether $ |\xi_{n+1}| $ is greater or less than 
$ (\lambda - \delta)(n+1)$ for some small $ \delta > 0$ to be chosen later.
\begin{align}
\label{expect_000}
&\tilde{E} \Big( e^{ \beta \tilde{L}_{n+1}} \mathbf{1}_{\{ \xi_{n+1} + \bar{ \xi}_n \ \geq \ \lambda n\}} \Big) 
e^{-(\Delta_{\lambda} + \epsilon)n} \nonumber \\
= &\tilde{E} \Big( e^{ \beta \tilde{L}_{n+1}} \mathbf{1}_{\{ \xi_{n+1} + \bar{ \xi}_n \ \geq \ \lambda n\}} 
\mathbf{1}_{\{ |\xi_{n+1}| \ > \ (\lambda - \delta)(n + 1) \}} \Big) e^{-(\Delta_{\lambda} + \epsilon)n} \nonumber \\
+ &\tilde{E} \Big( e^{ \beta \tilde{L}_{n+1}} \mathbf{1}_{\{ \xi_{n+1} + \bar{ \xi}_n \ \geq \ \lambda n\}} 
\mathbf{1}_{\{ |\xi_{n+1}| \ \leq \ (\lambda - \delta)(n + 1) \}} \Big) e^{-(\Delta_{\lambda} + \epsilon)n} \text{.}
\end{align}
Then from Theorem \ref{E_Nlambda} we have 
\begin{align*}
&\frac{1}{n} \log \Big( \tilde{E} \Big( e^{ \beta \tilde{L}_{n+1}} \mathbf{1}_{\{ \xi_{n+1} + \bar{ \xi}_n \ \geq \ \lambda n\}} 
\mathbf{1}_{\{ |\xi_{n+1}| \ > \ (\lambda - \delta)(n + 1) \}} \Big) e^{-(\Delta_{\lambda} + \epsilon)n} \Big) \\
\leq &\frac{1}{n} \log \Big( \tilde{E} \Big( e^{ \beta \tilde{L}_{n+1}} \mathbf{1}_{\{ |\xi_{n+1}| \ > \ (\lambda - \delta)(n + 1) \}}
\Big) e^{-(\Delta_{\lambda} + \epsilon)n} \Big) \\
= &\frac{1}{n} \log \Big( 2 \tilde{E} \Big( e^{ \beta \tilde{L}_{n+1}} \mathbf{1}_{\{ \xi_{n+1} \ > \ (\lambda - \delta)(n + 1) \}}
\Big) \Big) - (\Delta_{\lambda} + \epsilon) \\
\to & \Delta_{\lambda - \delta} - (\Delta_{\lambda} + \epsilon) \text{.}
\end{align*}
Since $ \Delta_{\lambda}$ is continuous in $ \lambda$, $ \Delta_{\lambda - \delta} - (\Delta_{\lambda} + \epsilon) < 0$ for 
$ \delta$ chosen small enough and hence the first expectation in \eqref{expect_000} decays exponentially fast. For the second expectation 
we have the following:
\begin{align*}
&\tilde{E} \Big( e^{ \beta \tilde{L}_{n+1}} \mathbf{1}_{\{ \xi_{n+1} + \bar{ \xi}_n \ \geq \ \lambda n\}} 
\mathbf{1}_{\{ |\xi_{n+1}| \ \leq \ (\lambda - \delta)(n + 1) \}} \Big) e^{-(\Delta_{\lambda} + \epsilon)n} \\
= &E^{\tilde{Q}_{\beta}} \Big( e^{ \beta | \xi_{n+1}| + \frac{1}{2} \beta^2(n+1)} \mathbf{1}_{\{ \xi_{n+1} + \bar{ \xi}_n \ \geq \ \lambda n\}} 
\mathbf{1}_{\{ |\xi_{n+1}| \ \leq \ (\lambda - \delta)(n + 1) \}} \Big) e^{-(\Delta_{\lambda} + \epsilon)n} \\
\leq &C E^{\tilde{Q}_{\beta}} \Big( \mathbf{1}_{\{ \xi_{n+1} + \bar{ \xi}_n \ \geq \ \lambda n\}} 
\mathbf{1}_{\{ |\xi_{n+1}| \ \leq \ (\lambda - \delta)(n + 1) \}} \Big) e^{K n} \\
\leq &C E^{\tilde{Q}_{\beta}} \Big( \mathbf{1}_{\{ \bar{ \xi}_n \ \geq \ \delta n + (\delta - \lambda) \}} \Big) e^{K n} \\
= &C  \tilde{Q}_{\beta} \big( \bar{ \xi}_1 \ \geq \ \delta n + (\delta - \lambda) \big) e^{K n} \text{,} 
\end{align*}
where $C = e^{\frac{1}{2} \beta^2 + (\lambda - \delta)}$ and $ K = \frac{1}{2} \beta^2 + \beta( \lambda - \delta) - ( \Delta_{\lambda} + \epsilon)$. 
However $\tilde{Q}_{\beta} \big( \bar{ \xi}_1 \ \geq \ \delta n + (\delta - \lambda) \big)$ decays faster than exponentially in $n$.  
To see this observe that for any $ \theta$ arbitrarily large
\[
\tilde{Q}_{\beta} \big( \bar{ \xi}_1 \ \geq \ \delta n \big) \leq E^{\tilde{Q}_{\beta}} \Big( e^{ \theta \bar{ \xi}_1 } \Big) e^{- \theta \delta n} \text{,}
\]
where 
\begin{align*}
E^{\tilde{Q}_{\beta}} \Big( e^{ \theta \bar{ \xi}_1 } \Big) &= \tilde{E} \Big( e^{ \theta \bar{ \xi}_1 }
e^{- \beta | \xi_1| + \beta \tilde{L}_1 - \frac{1}{2} \beta^2} \Big) \leq 
\tilde{E} \Big( e^{ \theta \bar{ \xi}_1 + \beta \tilde{L}_1 } \Big) \\
&\leq \Big( e^{2 \theta \bar{ \xi}_1} \Big)^{\frac{1}{2}} \Big( e^{2 \beta \tilde{L}_1} \Big)^{\frac{1}{2}} < \infty
\end{align*}
using the Cauchy-Schwarz inequality and the fact that $ \tilde{L}_1 \stackrel{d}{=} \bar{\xi}_1 \stackrel{d}{=} |N(0,1)|$ under $ \tilde{P}$. 
Thus we have shown that the expectation in \eqref{expect_000} and consequently $P(A_n)$ decay exponentially fast. 

So by the Borel-Cantelli lemma $P(A_n  \ i.o.) = 0$ and $P( A_n^c \ ev.) = 1$. That is,
\[
\sum_{u \in N_{n+1}} \mathbf{1}_{\{ \sup_{s \in [n,n+1]} X^u_s \ \geq \ \lambda n\}} 
\leq e^{(\Delta_{\lambda} + \epsilon)n} \text{ eventually.}
\]
So there exists a $P$-almost surely finite time $T_{\epsilon}$ such that $ \forall n > T_{\epsilon} $ 
\[
\sum_{u \in N_{n+1}} \mathbf{1}_{\{ \sup_{s \in [n,n+1]} X^u_s \ \geq \ \lambda n\}} 
\leq e^{(\Delta_{\lambda} + \epsilon)n} \text{.}
\]
Then 
\begin{align*}
&|N_t^{\lambda t}| \leq \sum_{u \in N_{ \lfloor t \rfloor + 1}} 
\mathbf{1}_{\{ \sup_{s \in [ \lfloor t \rfloor, \ \lfloor t \rfloor + 1]} X^u_s \ \geq \ \lambda \lfloor t \rfloor \}} \\
\Rightarrow &|N_t^{\lambda t}| \leq e^{(\Delta_{\lambda} + \epsilon) \lfloor t \rfloor } \quad \text{ for } t > T_{\epsilon} + 1 \text{,}
\end{align*}
which proves that 
\[
\limsup_{t \to \infty} \frac{\log |N_t^{\lambda t}|}{t} \leq \Delta_{\lambda} \quad P \text{-a.s.}
\]
\end{proof}
\begin{Remark}
\label{rem_logN}
Since $|N_t^{\lambda t}|$ takes only integer values we see that for $\lambda > \frac{\beta}{2}$ the inequality 
\[
\limsup_{t \to \infty} \frac{\log |N_t^{\lambda t}|}{t} \leq \Delta_{\lambda} < 0
\]
actually implies that $|N_t^{\lambda t}| \to 0 \ P$-a.s.
\end{Remark}
\subsection{Lower bound}
Before we present the proof of the lower bound of Theorem \ref{Nlambda_as} let us give a heuristic 
argument, which this proof will be based upon.

Take $ \lambda > 0$. Suppose we are given some large time $t$ and we want to estimate the number of 
particles $u \in N_t$ such that $|X^u_t| > \lambda t$. 

Let $p \in [0, 1]$. At time $pt$ the number of particles in the system is $|N_{pt}| \approx e^{\frac{1}{2} \beta^2 p t}$ by 
Theorem \ref{N_as}. If we ignore any branching that takes place in the time interval $(pt, t]$ then 
each of these particles will end up in the region $(- \infty , - \lambda t] \cup [\lambda t, \infty)$ at time $t$ 
with probability $ \gtrsim e^{-\frac{\lambda^2}{2(1-p)}t}$ using the standard estimate of the tail distribution of a normal 
random variable. 

Thus a crude estimate gives us that the number of particles at time $t$ in the region 
$(- \infty , - \lambda t] \cup [\lambda t, \infty)$ is 
\[
\gtrsim e^{-\frac{\lambda^2}{2(1-p)}t} \times |N_{pt}| \approx e^{-\frac{\lambda^2}{2(1-p)}t} \times e^{\frac{1}{2} \beta^2 p t} \text{.}
\]
The value of $p$ which maximises this expression is
\begin{equation*}
p^{\ast} = \left\{
\begin{array}{rl}
0 & \text{if } \lambda \geq \beta \\ 
1 - \frac{\lambda}{\beta} & \text{if } \lambda < \beta 
\end{array} \right.
\end{equation*}
and then 
\begin{equation*}
\frac{\log \Big( e^{-\frac{\lambda^2}{2(1-p)}t} \times |N_{pt}| \Big)}{t} \Big\vert_{p = p^{\ast}} \sim \left\{
\begin{array}{rl}
-\frac{1}{2} \lambda^2 & \text{if } \lambda \geq \beta \\ 
\frac{1}{2} \beta^2 - \beta \lambda & \text{if } \lambda < \beta 
\end{array} \right. = \Delta_{\lambda} \text{.}
\end{equation*}
Let us now use this idea to give a formal proof of the following lemma. 
\begin{Lemma}
\label{Nlambda_lower}
Take $ \lambda < \frac{\beta}{2}$. Then 
\[
\liminf_{t \to \infty} \frac{\log |N_t^{\lambda t}|}{t} \geq \Delta_{\lambda} = \frac{1}{2} \beta^2 - 
\beta \lambda \quad P \text{-a.s.}
\] 
\end{Lemma}
\begin{proof}
Take $p := 1 - \frac{\lambda}{\beta} \in (\frac{1}{2}, 1)$. For integer 
times $n$ we shall consider particles alive at time $p n$ (that is, particles 
in the set $N_{p n}$).

For each particle $u \in N_{pn}$ we can choose one descendant alive at time $n + 1$. 
Let $ \hat{N}_{n+1}$ be a set of such descendants (so that 
$| \hat{N}_{n + 1}| = |N_{pn}|$). \newline Then, for $u \in \hat{N}_{n+1}$, paths 
$ \big( X_t^u \big)_{t \in [pn, \ n+1]} $ correspond to independent Brownian 
motions (started at some unknown positions at time $p n$). Note that, wherever particle 
$u$ is at time $p n$,
\[
P \Big(  |X^u_s| > \lambda s \ \forall s \in [n, n+1] \Big) \gtrsim 
e^{- \frac{ \lambda^2}{2(1-p)} n} = e^{- \frac{1}{2} \beta \lambda n} =: q_n( \lambda)
\]
using the tail estimate of the normal distribution. Take a small $ \delta > 0$ to be specified later.
Then by Theorem \ref{N_as} 
\[
| \hat{N}_{n+1}| = |N_{pn}| \geq e^{(\frac{1}{2} \beta^2 p - \delta)n}  \text{ eventually.} 
\]
To prove Lemma \ref{Nlambda_lower} we fix an arbitrary $ \epsilon > 0$ and consider the events
\[
B_n := \Big\{ \sum_{u \in \hat{N}_{n + 1}} \mathbf{1}_{\{ |X^u_s| > \lambda s \ \forall s \in [n, n+1]  \}} 
< e^{(\Delta_{\lambda} - \epsilon)n} \Big\} \text{.}
\]
We wish to show that $P(B_n \ i.o.) = 0$. Now, 
\begin{align*}
&P \Big( B_n \cap \big\{ |\hat{N}_{n+1}| > e^{(\frac{1}{2} \beta^2 p - \delta)n} \big\} \Big) \\ 
= &P \Big( \Big\{ \sum_{u \in \hat{N}_{n + 1}} \mathbf{1}_{\{ |X^u_s| > \lambda s \ \forall s \in [n, n+1]  \}} 
< e^{(\Delta_{\lambda} - \epsilon)n} \Big\} \cap \big\{ |\hat{N}_{n+1}| > e^{(\frac{1}{2} \beta^2 p - \delta)n} \big\} \Big) \\ 
\leq &P \bigg( \sum_{i = 1}^{e^{(\frac{1}{2} \beta^2 p - \delta)n}} \mathbf{1}_{A_i} < e^{(\Delta_{\lambda} - \epsilon)n} \bigg) \text{,}
\end{align*}
where $A_i$'s are independent events with $P(A_i) \geq q_n(\lambda) \ \forall i$ \text{.} Then 
\begin{align*}
P \bigg( \sum_{i = 1}^{e^{(\frac{1}{2} \beta^2 p - \delta)n}} \mathbf{1}_{A_i} < e^{(\Delta_{\lambda} - \epsilon)n} \bigg) 
&= P \Big( e^{- \sum \mathbf{1}_{A_i}} > e^{- e^{(\Delta_{\lambda} - \epsilon)n}} \Big) \\
&\leq e^{e^{(\Delta_{\lambda} - \epsilon)n}} E \Big( e^{- \sum \mathbf{1}_{A_i}} \Big) \\
&= e^{e^{(\Delta_{\lambda} - \epsilon)n}} \prod_{i = 1}^{e^{(\frac{1}{2} \beta^2 p - \delta)n}} E \Big( e^{- \mathbf{1}_{A_i}} \Big) \\
&\leq e^{e^{(\Delta_{\lambda} - \epsilon)n}} \prod \Big( 1 - P(A_i) (1 - e^{-1}) \Big) \\
&\leq e^{e^{(\Delta_{\lambda} - \epsilon)n}} \prod \Big( 1 - q_n(\lambda) (1 - e^{-1}) \Big) \\
&\leq e^{e^{(\Delta_{\lambda} - \epsilon)n}} \prod e^{- q_n(\lambda) (1 - e^{-1})} \\
&= \exp \Big\{ e^{(\Delta_{\lambda} - \epsilon)n} - (1 - e^{-1}) q_n( \lambda) e^{(\frac{1}{2} \beta^2 p - \delta)n} \Big\} \\
&= \exp \Big\{ e^{(\Delta_{\lambda} - \epsilon)n} - (1 - e^{-1}) e^{(\Delta_{\lambda} - \delta )n} \Big\} \text{.}
\end{align*}
This expression decays fast enough if we take $\delta < \epsilon$. Thus 
\[
P \Big( B_n \cap \big\{ |\hat{N}_{n+1}| > e^{(\frac{1}{2} \beta^2 p - \delta)n} \big\} \text{ i.o.} \Big) = 0 \text{.}
\]
And since  $P \Big( \big\{ |\hat{N}_{n+1}| > e^{(\frac{1}{2} \beta^2 p - \delta)n} \big\} $ 
ev. $ \Big) = 1$, we get that $ P \Big( B_n $  i.o. $ \Big) = 0 $. That is, 
\[
\sum_{u \in \hat{N}_{n + 1}} \mathbf{1}_{\{ |X^u_s| > \lambda s \ \forall s \in [n, n+1]  \}} 
\geq e^{(\Delta_{\lambda} - \epsilon)n} \text{ for } n \text{ large enough } P \text{-almost surely.}
\]
Now, since the process is symmetric, the probability that a particle 
$u \in \hat{N}_{n+1}$ such that $|X^u_s| > \lambda s \ \forall s \in [n, n+1]$ 
actually satisfies $X^u_s > \lambda s \ \forall s \in [n, n+1]$ is $ \frac{1}{2}$. So 
applying the usual Borel-Cantelli argument once again we can for example prove that for 
some constant $C \in (0, \frac{1}{2})$
\[
\sum_{u \in \hat{N}_{n + 1}} 
\mathbf{1}_{\{ X^u_s > \lambda s \ \forall s \in [n, n+1]  \}} 
\geq C e^{(\Delta_{\lambda} - \epsilon)n} \text{ for } n \text{ large enough } P \text{-almost surely.}
\]
Then for $t$ large enough 
\[
|N_t^{\lambda t}| = \sum_{u \in N_t} \mathbf{1}_{\{ X^u_t > \lambda t\}} \geq \sum_{u \in \hat{N}_{\lfloor t \rfloor + 1}} 
\mathbf{1}_{\{ X^u_s > \lambda s \ \forall s \in [\lfloor t \rfloor , \lfloor t \rfloor + 1]  \}} 
\geq C e^{(\Delta_{\lambda} - \epsilon) \lfloor t \rfloor} \text{.}
\]
Thus
\[
\liminf_{t \to \infty} \frac{ \log |N_t^{\lambda t}|}{t} \geq \Delta_{ \lambda} \text{.}
\]
\end{proof}
Lemmas \ref{Nlambda_upper} and \ref{Nlambda_lower} together prove Theorem \ref{Nlambda_as}. 
\subsection{Decay of $P(|N_t^{\lambda t}| \geq 1)$ in the case $ \lambda > \frac{\beta}{2}$} 
Theorem \ref{Nlambda_as} told us that if $ \lambda > \frac{\beta}{2}$ then 
$|N_t^{\lambda t}| \to 0$. Let us also prove that in this case
\[
\frac{\log P(|N_t^{\lambda t}| \geq 1)}{t} \to \Delta_{\lambda} = - \frac{1}{2} \lambda^2 \text{.}
\]
\begin{proof}[Proof of Lemma \ref{log_P}]
Trivially $P(|N_t^{\lambda t}| \geq 1) \leq E |N_t^{ \lambda t}|$. Hence by Theorem \ref{E_Nlambda}
\[
\limsup_{t \to \infty} \frac{\log P(|N_t^{\lambda t}| \geq 1)}{t} \leq \Delta_{ \lambda} \text{.}
\]
For the lower bound we use the same idea as in Lemma \ref{Nlambda_lower}. Let us take 
\begin{equation*}
p = \left\{
\begin{array}{rl}
0 & \text{if } \lambda > \beta \\
1 - \frac{\lambda}{\beta} & \text{if } \lambda \leq \beta 
\end{array} \right.
\end{equation*}
We define a set $\hat{N}_t$ as in Subsection 6.2 That is, for each particle 
$u \in N_{pt}$ we choose one descendent alive at time $t$ (so that 
$\hat{N}_t \subset N_t$, $| \hat{N}_t| = |N_{pt}|$). Then for each 
$u \in \hat{N}_t$ wherever it is at time $pt$ we have 
\[
P(|X^u_t| > \lambda t) \gtrsim e^{- \frac{\lambda^2}{2(1 - p)} t} =: p_t( \lambda) \text{.}
\]
Then
\[
P \Big( |N_t^{\lambda t}| \geq 1 \Big) \geq \frac{1}{2} P \Big( |N_t^{ \pm \lambda t}| \geq 1 \Big) \text{,}
\]
where  $N_t^{\pm \lambda t} := \{ u \in N_t : |X^u_t| > \lambda t\}$. Thus for a small $ \delta > 0 $ to be specified 
later we have 
\begin{align*}
P \Big( |N_t^{\lambda t}| \geq 1 \Big) &\geq \frac{1}{2} P \Big( |N_t^{ \pm \lambda t}| \geq 1, \ 
|N_{pt}| > \underbrace{e^{(\frac{1}{2} \beta^2 p - \delta)t}}_{:= n_t(\delta)} \Big) \\
&\geq \frac{1}{2} P \Big( \bigcup_{u \in \hat{N}_t} \{ |X^u_t| > \lambda t \}, \ | N_{pt}| > n_t( \delta) \Big) \\ 
&\geq \frac{1}{2} \Big( 1 - \big( 1 - p_t(\lambda) \big)^{n_t(\delta)} \Big) P \big( | N_{pt}| > n_t( \delta) \big) \text{.}
\end{align*}
By Theorem \ref{N_as}, $P \big( | N_{pt}| > n_t( \delta) \big) \to 1$, so we can ignore this term. Then 
\begin{align*}
\Big( 1 - \big( 1 - p_t(\lambda) \big)^{n_t(\delta)} \Big) 
&= n_t(\delta) p_t(\lambda) - \binom{n_t(\delta)}{2} p_t(\lambda)^2 + 
\binom{n_t(\delta)}{3} p_t(\lambda)^3 - \cdots \\
&\geq n_t(\delta) p_t(\lambda) - n_t(\delta)^2 p_t(\lambda)^2 \Big(
1 + n_t(\delta) p_t(\lambda) + n_t(\delta)^2 p_t(\lambda)^2 + \cdots \Big) \text{.}
\end{align*}
Note that for $ \delta$ small enough
\[
n_t(\delta) p_t(\lambda) = e^{(\frac{1}{2} \beta^2 p - \delta)t} e^{- \frac{\lambda^2}{2(1 - p)}t} = e^{(\Delta_{\lambda} - \delta)t} \ \ll 1 \text{.}
\] 
Hence $P(|N_t^{\lambda t}| \geq 1) \geq \Big( \frac{1}{2} P \big( | N_{pt}| > n_t ( \delta) \big) \Big) 
e^{(\Delta_{\lambda} - \delta)t} + o \big( e^{(\Delta_{\lambda} - \delta)t} \big) $ and therefore
\[
\liminf_{t \to \infty} \dfrac{\log P(|N_t^{\lambda t}| \leq 1)}{t} \geq \Delta_{ \lambda} \text{.}
\]
This completes the proof of Lemma \ref{log_P}.
\end{proof}
\subsection{The rightmost particle} 
Observe that the number of particles above the line $\lambda t$ grows exponentially if $ \lambda < \frac{\beta}{2}$ and is eventually 
$0$ if $ \lambda > \frac{ \beta}{2}$. Hence, as a corollary of Theorem \ref{Nlambda_as}, we get that
\[
\frac{R_t}{t} \to \frac{\beta}{2} \quad P \text{-a.s.,}
\]
where $(R_t)_{t \geq 0}$ is the rightmost particle of the branching process. 
\begin{proof}[Proof of Corollary \ref{Th_rightmost}]
Take any $ \lambda < \frac{\beta}{2}$. By Theorem \ref{Nlambda_as} $|N_t^{\lambda t}| \geq 1 \ \forall t$ large enough, 
so $R_t \geq \lambda t$ for $t$ large enough. Thus $\liminf_{t \to \infty} t^{-1} R_t \geq \lambda \ P$-a.s. 
Letting $ \lambda \nearrow \frac{\beta}{2}$ we get
\[
\liminf_{t \to \infty} \frac{R_t}{t} \geq \frac{\beta}{2} \quad P \text{-a.s.}
\]
Similarly, if we take $ \lambda > \frac{\beta}{2}$ then by Theorem \ref{Nlambda_as} $|N_t^{\lambda t}| = 0 \ \forall t$ 
large enough and so $R_t \leq \lambda t$ for $t$ large enough. Hence $\limsup_{t \to \infty} t^{-1} R_t \leq \lambda \ P$-a.s.
So, letting $ \lambda \searrow \frac{\beta}{2}$ we get
\[
\limsup_{t \to \infty} \frac{R_t}{t} \leq \frac{\beta}{2} \quad P \text{-a.s.}
\]
and this proves Corollary \ref{Th_rightmost}.
\end{proof}
Note that the rightmost particle (that is, the extremal particle) in our model behaves 
very differently from the rightmost particle in the model with homogeneous branching. 

In the BBM model with homogeneous branching rate $\beta$ there is a particle staying near the critical line $ \sqrt{2 \beta}t$ all 
the time. (The word particle here is a bit ambiguous since we are really talking about an infinite 
line of descent, but this is a common description.)

On the other hand in the BBM model with branching rate $ \beta \delta_0(x)$, since branching only takes place at the origin, 
no particle can stay close to any straight line $ \lambda t$, $ \lambda > 0$ 
all the time. The optimal way for some particle to reach the critical line $ \frac{\beta}{2} t$ at time $T$ 
is to wait near the origin until the time $\frac{T}{2}$ in order to give birth to as many particles as 
possible, and then at time $ \frac{T}{2}$ one of $ \approx e^{\frac{\beta^2}{4}T}$ particles will have a 
good chance of reaching $ \frac{\beta}{2}T$ at time $T$.     
\section{Strong law of large numbers}
Recall the additive martingale $M_t = e^{- \frac{ \beta^2}{2}t} \sum_{u \in N_t} 
e^{- \beta |X^u_t|}$, $t \geq 0$ from \eqref{addtive_Mbeta} and the measure $ \pi( \mathrm{d}x) = \beta e^{- 2 \beta |x|} 
\mathrm{d}x$ from Proposition \ref{prop_drift_to_0}. 
In this section, we shall discuss Theorem \ref{new_SLLN} which says that for a measurable 
bounded function $f(\cdot)$ 
\begin{align}
\label{llmm}
\lim_{t \to \infty} e^{- \frac{\beta^2}{2}t} \sum_{u \in N_t} f(X^u_t) &= 
M_{\infty} \int_{- \infty}^{\infty} f(x) \beta e^{- \beta |x|} \mathrm{d}x \nonumber \\
&= M_{\infty} \int_{- \infty}^{\infty} f(x) e^{ \beta |x|} \pi( \mathrm{d}x ) \ P \text{-a.s.}
\end{align}
Observe that convergence in $L^1$ is trivial by the Many-to-One theorem and the uniform integrability 
of $(M_t)_{t \geq 0}$:
\begin{align*}
 E \Big( e^{- \frac{ \beta^2}{2}t} \sum_{u \in N_t} f(X^u_t) \Big) &= 
\tilde{E} \Big( e^{- \frac{ \beta^2}{2}t} f(\xi_t) e^{\beta \tilde{L}_t} \Big) \\
&=\tilde{E} \Big( f(\xi_t) e^{\beta |\xi_t|} \big(  
e^{- \beta |\xi_t| + \beta \tilde{L}_t - \frac{ \beta^2}{2}t } \big) \Big) \\ 
&= E^{\tilde{Q}_{\beta}} \Big( f(\xi_t) e^{\beta |\xi_t|} \Big) \to \int f(x) e^{\beta |x|} 
\pi (\mathrm{d}x) \text{.}
\end{align*}
Also the Weak Law of Large Numbers for this model have been proved by J. Engl\"ander and D. Turaev in \cite{19}. 
In particular they have given the law of $M_{\infty}$.

As a simple corollary of Theorem \ref{new_SLLN} we get by taking $f(\cdot) \equiv 1$ that
\begin{equation}
\label{mkkm}
|N_t| e^{- \frac{1}{2} \beta^2 t} \to 2 M_{ \infty} \quad P \text{-a.s.}
\end{equation}
This should be compared with results in Lemma \ref{E_N} and Theorem \ref{N_as}.

Dividing \eqref{llmm} by \eqref{mkkm} we get an alternative form of Theorem \ref{new_SLLN}: 
\[
\dfrac{\sum_{u \in N_t} f(X^u_t)}{|N_t|} \to \frac{1}{2} \int f(x) e^{\beta |x|} \pi 
(\mathrm{d}x) = \frac{\beta}{2} \int f(x) e^{- \beta |x|} \mathrm{d}x \quad P
\text{-a.s.}
\]
The Strong Law of Large Numbers was proved in \cite{17} for a large class of general diffusion processes and 
branching rates $\beta(x)$. In our case the branching rate is a generalised 
function $ \beta \delta_0(x)$, which doesn't satisfy the conditions of \cite{17}. 
Nevertheless we can adapt the proof to our model if we take the generalised principal 
eigenvalue $\lambda_c = \frac{\beta^2}{2}$ and the eigenfunctions $ \phi(x) = e^{- \beta |x|}$, 
$\tilde{\phi}(x) = \beta e^{- \beta |x|}$ in \cite{17}. Also the proof relies on the $L^p$ convergence of the martingale $(M_t)_{t \geq 0}$ and the 
linear asymptotic growth of the rightmost particle which we have derived earlier in this article.

Let us now finish the article with the proof of Theorem \ref{new_SLLN}.
\begin{proof}[Proof of Theorem \ref{new_SLLN}]
Take $B$ a measurable set. As it will be shown later, it is sufficient to prove the theorem for functions of the form 
$f(x) = e^{- \beta |x|} \mathbf{1}_{ \{ x \in B \} }$. For such set $B$ let
\[
U_t := e^{- \frac{\beta^2}{2}t} \sum_{u \in N_t} 
e^{- \beta |X^u_t|} \mathbf{1}_{ \{ X^u_t \in B \} } = e^{- \frac{\beta^2}{2}t} \sum_{u \in N_t} f(X^u_t) 
\]
So if $B = \mathbb{R}$ then we would have $U_t = M_t$ and generally 
$U_t \leq M_t$. We wish to show that 
\[
U_t \to \pi(B) M_{\infty} \ \Big( = \int f(x) e^{ \beta |x|}
\pi ( \mathrm{d}x) M_{\infty} \Big)  \text{ as } t \to \infty \text{.}
\]
The proof can be split into three parts.

\underline{Part I}:

Let us take $K > 0$. At this stage it doesn't matter what $K$ is, but in Part II of the proof 
we shall choose an appropriate value for it. Let $m_n := Kn $ (using the same notation as in 
\cite{17}). Also fix $ \delta > 0$. We first want to prove that
\begin{equation}
\label{new_eq1}
\lim_{n \to \infty} \Big\vert U_{(n + m_n) \delta} - E \Big( U_{(n + m_n) \delta} | 
\mathcal{F}_{n \delta} \Big) \Big\vert = 0 \quad P \text{-a.s.}
\end{equation}
We begin with the observation that
\begin{equation}
\label{new_eq2}
\forall s,t \geq 0 \qquad U_{s + t} = \sum_{u \in N_t} e^{- \frac{\beta^2}{2} t} U_s^{(u)} \text{,}
\end{equation}
where conditional on $ \mathcal{F}_t$, $U_s^{(u)}$ are independent 
copies of $U_s$ started from the positions $X^u_t$.

To prove \eqref{new_eq1} using the Borel-Cantelli lemma we need to show that for all 
$ \epsilon > 0$
\begin{equation}
\label{new_eq3}
\sum_{n = 1}^{ \infty} P \Big( \Big\vert U_{(n + m_n) \delta} - E \Big( U_{(n + m_n) \delta} | 
\mathcal{F}_{n \delta} \Big) \Big\vert > \epsilon \Big) < \infty \text{.}
\end{equation}
Let us take any $p \in (1,2)$.Then 
\begin{align*}
P \Big( \Big\vert U_{(n + m_n) \delta} - 
E \Big( U_{(n + m_n) \delta} | \mathcal{F}_{n \delta} \Big) \Big\vert > \epsilon \Big) 
\leq \frac{1}{ \epsilon^p} E \Big( \Big\vert U_{(n + m_n) \delta} - 
E \Big( U_{(n + m_n) \delta} | \mathcal{F}_{n \delta} \Big) \Big\vert^p \Big) \text{.}
\end{align*}
Next we shall apply the following inequality, which was used in the proof of the SLLN in \cite{17} and can also be found 
in \cite{16}: if $p \in (1,2)$ and $X_i$ are independent random variables with $ \mathbb{E}X_i = 0$ (or they are martingale differences), 
then 
\begin{equation}
\label{ten}
\mathbb{E} \Big\vert \sum_{i = 1}^n X_i \Big\vert^p \leq 2^p \sum_{i = 1}^n \mathbb{E} |X_i|^p \text{.}
\end{equation}
Then by \eqref{new_eq2} 
\[
U_{s+t} - E \Big( U_{s+t} | \mathcal{F}_t \Big) = \sum_{u \in N_t} e^{- \frac{\beta^2}{2} t} \Big( U_s^{(u)} - 
E \big( U_s^{(u)} | \mathcal{F}_t \big) \Big) \text{,}
\]
where conditional on $\mathcal{F}_t$, $U_s^{(u)} - E \big( U_s^{(u)} | \mathcal{F}_t \big)$ are independent 
with $0$ mean. Thus applying \eqref{ten} and Jensen's inequality we get 
\begin{align}
\label{new_eq4}
E \Big( \Big\vert U_{s + t} - E \Big( U_{s + t} | \mathcal{F}_t \Big) \Big\vert^p | \mathcal{F}_t \Big) 
\leq &2^p e^{- p \frac{\beta^2}{2} t} \sum_{u \in N_t} E \Big( \Big\vert U_s^{(u)} - E(U_s^{(u)} | \mathcal{F}_t) 
\Big\vert^p | \mathcal{F}_t \Big) \nonumber \\
\leq &2^p e^{- p \frac{\beta^2}{2} t} \sum_{u \in N_t} E \Big( 2^{p - 1} \Big( \big\vert U_s^{(u)} 
\big\vert^p + \big\vert E(U_s^{(u)} | \mathcal{F}_t) \big\vert^p \Big) | \mathcal{F}_t \Big) \nonumber \\
\leq &2^p e^{- p \frac{\beta^2}{2} t} \sum_{u \in N_t} E \Big( 2^{p - 1} \Big( \big\vert U_s^{(u)} \big\vert^p + 
E \big( \big\vert U_s^{(u)} \big\vert^p | \mathcal{F}_t \big) \Big) | \mathcal{F}_t \Big) \nonumber \\
= &2^{2p} e^{- p \frac{\beta^2}{2} t} \sum_{u \in N_t} E \Big( \big\vert U_s^{(u)} \big\vert^p | \mathcal{F}_t \Big) \text{.}
\end{align}
Hence by \eqref{new_eq4}
\begin{align*}
\sum_{n = 1}^{\infty} E \Big( \Big\vert U_{(n + m_n) \delta} - 
E \Big( U_{(n + m_n) \delta} | \mathcal{F}_{n \delta} \Big) \Big\vert^p \Big) 
&\leq 2^{2 p} \sum_{n = 1}^{\infty} e^{- p \frac{\beta^2}{2} \delta n}
E \Big( \sum_{u \in N_{\delta n}} E^{X_{\delta n}^u} \big( U_{m_n \delta} \big)^p \Big) \\
&\leq 2^{2 p} \sum_{n = 1}^{\infty} e^{- p \frac{\beta^2}{2} \delta n}
E \Big( \sum_{u \in N_{\delta n}} E^{X_{\delta n}^u} \big( M_{m_n \delta} \big)^p \Big) \\
&= 2^{2 p} \sum_{n = 1}^{\infty} e^{- p \frac{\beta^2}{2} \delta n}
E \Big( \sum_{u \in N_{\delta n}} e^{-\beta p |X_{\delta n}^u|} E^0 \big( M_{m_n \delta} \big)^p \Big) \\
&\leq \sum_{n = 1}^{\infty} e^{- p \frac{\beta^2}{2} \delta n} 
e^{\frac{\beta^2}{2} \delta n} \times C \text{,}
\end{align*}
where $C$ is some positive constant and we have applied the Many-to-One Theorem (Theorem \ref{manytoone_0}) and 
and Theorem \ref{Lp} in the last inequality. Since $p > 1$ the sum is $< \infty$. This finishes the proof of \eqref{new_eq3} 
and hence \eqref{new_eq1}.

\underline{Part II}:

Let us now prove that
\begin{equation}
\label{new_eq5}
\lim_{n \to \infty} \Big\vert E \Big( U_{(n + m_n) \delta} | 
\mathcal{F}_{n \delta} \Big) - \pi(B) M_{\infty} \Big\vert = 0 \quad P \text{-a.s.}
\end{equation}
Together with \eqref{new_eq1} this will complete the proof of Theorem \ref{new_SLLN} along lattice times 
for functions $f(x)$ of the form $e^{- \beta |x|} \mathbf{1}_{ \{ x \in B \} }$.

We begin by noting that
\begin{align*}
E \Big( U_{s+t} | \mathcal{F}_t \Big) = &E \Big( \sum_{u \in N_t} e^{- \frac{\beta^2}{2} t} U_s^{(u)} | \mathcal{F}_t \Big) \\
= &\sum_{u \in N_t} e^{- \frac{\beta^2}{2} t} E^{X^u_t} U_s \\
= &\sum_{u \in N_t} e^{- \frac{\beta^2}{2} t} E^{X^u_t} \Big( \sum_{u \in N_s} e^{- \frac{\beta^2}{2} s - \beta |X^u_s|} 
\mathbf{1}_{ \{ X^u_s \in B \} } \Big) \\
= &\sum_{u \in N_t} e^{- \frac{\beta^2}{2} t} \tilde{E}^{X^u_t} \Big( e^{- \frac{\beta^2}{2} s - \beta | \xi_s|}
\mathbf{1}_{ \{ \xi_s \in B \} } e^{ \beta \tilde{L}_s} \Big) \\
= &\sum_{u \in N_t} e^{- \frac{\beta^2}{2} t - \beta |X^u_t|} E^{ \tilde{Q}_{\beta}^{X^u_t}} \Big( \mathbf{1}_{ \{ \xi_s \in B \} } \Big) \\
= &\sum_{u \in N_t} e^{- \frac{\beta^2}{2} t - \beta |X^u_t|} \int_B p(s, X^u_t, y) m( \mathrm{d} y) \text{,}
\end{align*}
where $ \tilde{Q}_{\beta}$ and $p(\cdot)$ were defined in \eqref{Q_beta0} and Proposition \ref{prop_drift_to_0}. Thus 
\begin{equation}
\label{new_eq6}
E \Big( U_{(n + m_n) \delta} | \mathcal{F}_{n \delta} \Big) = \sum_{u \in N_{n \delta}} 
e^{- \frac{\beta^2}{2} n \delta - \beta |X^u_{n \delta}|} \int_B p(m_n \delta, X^u_{n \delta}, y) m( \mathrm{d} y) \text{.}
\end{equation}
Recalling that $m_n = Kn$ where $K > 0$ we have
\[
E \Big( U_{(n + m_n) \delta} | \mathcal{F}_{n \delta} \Big) = \sum_{u \in N_{n \delta}} 
e^{- \frac{\beta^2}{2} n \delta - \beta |X^u_{n \delta}|} \int_B p(Kn \delta, X^u_{n \delta}, y) m( \mathrm{d} y) \text{.}
\]
Now choose $M > \frac{\beta}{2}$ and consider events
\[
C_n := \Big\{ |X^u_{n \delta}| < Mn \delta \quad \forall u \in N_{n \delta} \Big\} \text{.}
\]
Then 
\begin{align*}
&\sum_{u \in N_{n \delta}} e^{- \frac{\beta^2}{2} n \delta - \beta |X^u_{n \delta}|} \int_B p(Kn \delta, X^u_{n \delta}, y) m( \mathrm{d} y) \\
= &\sum_{u \in N_{n \delta}} e^{- \frac{\beta^2}{2} n \delta - \beta |X^u_{n \delta}|} \int_B p(Kn \delta, X^u_{n \delta}, y) 
m( \mathrm{d} y) \ \mathbf{1}_{ C_n^c} \\
& \quad + \sum_{u \in N_{n \delta}} e^{- \frac{\beta^2}{2} n \delta - \beta |X^u_{n \delta}|} \int_B p(Kn \delta, X^u_{n \delta}, y) 
m( \mathrm{d} y) \ \mathbf{1}_{ C_n } \text{.}
\end{align*}
The first sum is $0$ for $n$ large enough by Corollary \ref{Th_rightmost} (or even earlier by Theorem 
\ref{Nlambda_as}). To deal with the second sum we substitute the known transition density $p(\cdot)$: 
\begin{align*}
&\int_B p(Kn \delta, X^u_{n \delta}, y) m( \mathrm{d} y) \ \mathbf{1}_{ C_n } \\
= &\int_B \frac{1}{\sqrt{2 \pi K n \delta}} \exp \bigg\{ \beta \Big( |X^u_{n \delta}| - |y| \Big) - 
\frac{\beta^2}{2} Kn \delta - \frac{(X^u_{n \delta} - y)^2}{2Kn \delta} \bigg\} \\ 
& \quad + \frac{\beta}{2} Erfc \bigg( \frac{|X^u_{n \delta}| + |y| - \beta Kn \delta}{ \sqrt{2Kn \delta}} \bigg) 
e^{-2 \beta |y|} \ \mathrm{d}y \ \mathbf{1}_{ C_n } \text{.}
\end{align*}
Then for any given $M > \frac{ \beta}{2}$ we can choose $K > \frac{2M}{\beta}$ and hence 
\begin{align*}
&\int_B \ \frac{1}{\sqrt{2 \pi K n \delta}} \exp \bigg\{ \beta \Big( |X^u_{n \delta}| - |y| \Big) 
- \frac{\beta^2}{2} Kn \delta - \frac{(X^u_{n \delta} - y)^2}{2Kn \delta} \bigg\} \ \mathrm{d}y \ \mathbf{1}_{ C_n } \\
\leq &\exp \Big\{ \Big( \beta M - \frac{\beta^2}{2} K \Big) n \delta \Big\} \times C' \ \to 0 \text{ as } n \to \infty \text{,}
\end{align*}
where $C'$ is some positive constant and
\begin{align*}
\int_B \frac{\beta}{2} Erfc \bigg( \frac{|X^u_{n \delta}| + |y| - \beta Kn \delta}{ \sqrt{2Kn \delta}} \bigg) 
e^{-2 \beta |y|} \ \mathrm{d}y \ \mathbf{1}_{ C_n } 
\to \int_B \beta e^{-2 \beta |y|} \mathrm{d}y = \pi(B) \quad \text{as } n \to \infty
\end{align*}
since $Erfc(x) \to 2$ as $x \to - \infty$ and $\mathbf{1}_{C_n} \to 1$ as $n \to \infty$. 
Then going back to \eqref{new_eq6} and  we see that
\[
\lim_{n \to \infty} \Big\vert E \Big( U_{(n + m_n) \delta} | 
\mathcal{F}_{n \delta} \Big) - \pi(B) M_{n \delta} \Big\vert = 0 \quad P \text{-a.s.}
\]
and so also
\[
\lim_{n \to \infty} \Big\vert E \Big( U_{(n + m_n) \delta} | 
\mathcal{F}_{n \delta} \Big) - \pi(B) M_{ \infty} \Big\vert = 0 \quad P \text{-a.s.}
\]
As it was mentioned earlier parts I and II together complete the proof of Theorem \ref{new_SLLN} 
along lattice times for functions of the form $f(x) = e^{- \beta |x|} \mathbf{1}_B(x)$. To see this put together \eqref{new_eq1} 
and \eqref{new_eq5} to get that
\[
\lim_{n \to \infty} \Big\vert U_{(n + m_n) \delta} - \pi(B) M_{\infty} 
\Big\vert = 0 \quad P \text{-a.s.}
\]
That is, 
\[
\lim_{n \to \infty} \Big\vert U_{n(K + 1) \delta} - \pi(B) M_{\infty}
\Big\vert = 0 \quad P \text{-a.s.}
\]
Then $K + 1$ can be absorbed into $ \delta$ which stayed arbitrary throughout the proof. 
Also as it was mentioned earlier we can easily replace functions of the form 
$e^{- \beta |x|} \mathbf{1}_B(x)$ with any measurable functions.To see this we note that given any 
meausurable set $A$ and $ \epsilon_1 > 0$ we can find constants $\underline{c}_1$, $ \dots$ , $\underline{c}_n$, 
$ \bar{c}_1$, $ \dots$, $ \bar{c}_n$ and measurable sets $A_1$, $ \dots$, $A_n$ such that 
\[
\Big( \sum_{i = 1}^n \bar{c}_i \mathbf{1}_{A_i}(x) e^{- \beta |x|} \Big) - \epsilon_1 \leq \mathbf{1}_A(x) 
\leq \sum_{i = 1}^n \bar{c}_i \mathbf{1}_{A_i}(x) e^{- \beta |x|}
\]
and
\[
\sum_{i = 1}^n \underline{c}_i \mathbf{1}_{A_i}(x) e^{- \beta |x|} \leq \mathbf{1}_A(x) 
\leq \Big( \sum_{i = 1}^n \underline{c}_i \mathbf{1}_{A_i}(x) e^{- \beta |x|} \Big) + \epsilon_1 \text{.}
\]
Similarly given any positive bounded measurable function $f$ and $ \epsilon_2 > 0$ we can 
find simple functions $ \underline{f}$ and $ \bar{f}$ such that 
\[
\bar{f}(x) - \epsilon_2 \leq f(x) \leq \bar{f}(x)
\]
and 
\[
\underline{f}(x) \leq f(x) \leq \underline{f}(x) + \epsilon_2 \text{.}
\]
Thus given any positive bounded measurable function $f$ and $ \epsilon > 0$ we can 
find functions $ \bar{f}^{\epsilon}(x)$ and $ \underline{f}^{\epsilon}(x)$, which are linear 
combinations of functions of the form $ e^{- \beta |x|} \mathbf{1}_A(x)$ such that 
\[
\bar{f}^{\epsilon}(x) - \epsilon \leq f(x) \leq \bar{f}^{\epsilon}(x)
\]
and 
\[
\underline{f}^{\epsilon}(x) \leq f(x) \leq \underline{f}^{\epsilon}(x) + \epsilon \text{.}
\]
Then 
\begin{align*}
&\bar{f}^{\epsilon}(x) \beta e^{- \beta |x|} \leq (f(x) + \epsilon) \beta e^{- \beta |x|} \\
\Rightarrow &\int_{-\infty}^{\infty} \bar{f}^{\epsilon}(x) \beta e^{- \beta |x|} \mathrm{d}x 
\leq \int_{-\infty}^{\infty} (f(x) + \epsilon) \beta e^{- \beta |x|} \mathrm{d}x + 2 \epsilon 
\end{align*}
and hence $P$-almost surely we have 
\begin{align*}
\limsup_{n \to \infty} e^{- \frac{\beta^2}{2} n \delta} \sum_{u \in N_{n \delta}} f(X^u_{n \delta}) &\leq 
\limsup_{n \to \infty} e^{- \frac{\beta^2}{2} n \delta} \sum_{u \in N_{n \delta}} \bar{f}^{\epsilon}(X^u_{n \delta}) \\ 
&= M_{\infty} \int_{-\infty}^{\infty} \bar{f}^{\epsilon}(x) \beta e^{- \beta |x|} \mathrm{d}x \\
&\leq M_{\infty} \Big( \int_{-\infty}^{\infty} f(x) \beta e^{- \beta |x|} \mathrm{d}x + 2 \epsilon \Big) \text{.}
\end{align*}
Since $\epsilon$ is arbitrary we get 
\[
\limsup_{n \to \infty} e^{- \frac{\beta^2}{2} n \delta} \sum_{u \in N_{n \delta}} f(X^u_{n \delta}) \leq 
M_{\infty} \int_{-\infty}^{\infty} f(x) \beta e^{- \beta |x|} \mathrm{d}x \text{.}
\]
Similarly
\[
\liminf_{n \to \infty} e^{- \frac{\beta^2}{2} n \delta} \sum_{u \in N_{n \delta}} f(X^u_{n \delta}) \geq 
M_{\infty} \int_{-\infty}^{\infty} f(x) \beta e^{- \beta |x|} \mathrm{d}x \text{.}
\]
Also any bounded measurable function $f$ can be written as a difference of two positive 
bounded measurable functions. This completes the proof of Theorem \ref{new_SLLN} with the limit 
taken along lattice times. Now let us finish the proof of the theorem by extending it 
to the continuous-time limit. 

\underline{Part III}:

As in the previous parts of the proof it is sufficient to consider functions of the form 
$f(x) = e^{- \beta |x|} \mathbf{1}_B(x)$ for measurable sets $B$.

Let us now take $ \epsilon > 0$ and define the following set 
\[
B^{\epsilon}(x) := B \cap \Big( -|x| - \frac{1}{\beta} \log (1 + \epsilon ) , \
|x| + \frac{1}{\beta} \log (1 + \epsilon ) \Big) \text{.}
\]
Note that $y \in B^{\epsilon}(x)$ iff $y \in B$ and $e^{- \beta |y|} > \frac{e^{- \beta |x|}}{1 + \epsilon}$. 
Furthermore, for $\delta, \epsilon > 0$ let
\[
\Xi_B^{\delta, \epsilon}(x) := 
\mathbf{1}_{ \{ X^u_s \in B^{\epsilon}(x) \ \forall s \in [0, \delta] \ \forall u \in N_{\delta} \} }
\]
and
\[
\xi_B^{\delta, \epsilon}(x) := E^x \Big( \Xi_B^{\delta, \epsilon}(x) \Big) \text{.}
\]
Then for $t \in \ [n \delta, \ (n+1) \delta]$
\begin{align}
\label{new_eq7}
U_t &=  e^{- \frac{\beta^2}{2}t} \sum_{u \in N_t} e^{- \beta |X^u_t|} \mathbf{1}_{ \{ X^u_t \in B \} } \nonumber \\
&= \sum_{u \in N_{n \delta}} e^{- \frac{ \beta^2}{2} n \delta} U^{(u)}_{t - n \delta} \geq e^{- \frac{ \beta^2}{2} n \delta} 
\sum_{u \in N_{n \delta}} U^{(u)}_{t - n \delta} \ \Xi_B^{ \delta , \epsilon} (X^u_{n \delta}) \nonumber \\
&\geq e^{- \frac{ \beta^2}{2} n \delta} \sum_{u \in N_{n \delta}} e^{- \frac{\beta^2}{2} \delta} \
\frac{e^{- \beta |X^u_{n \delta}|}}{1 + \epsilon} \ \Xi_B^{ \delta , \epsilon} (X^u_{n \delta})
\end{align}
because at time $t$ there is at least one descendent of each particle alive at time 
$n \delta $. Let us consider the sum
\[
e^{- \frac{ \beta^2}{2} n \delta} \sum_{u \in N_{n \delta}} e^{- \beta |X^u_{n \delta}|} 
\Xi_B^{ \delta , \epsilon} (X^u_{n \delta}) \text{.}
\]
Note that 
\begin{equation}
\label{new_eq8}
\Xi_B^{ \delta , \epsilon} (X^u_{n \delta} ) 
\text{ are independent conditional on } \mathcal{F}_{n \delta} \text{,}
\end{equation}
\begin{equation}
\label{new_eq9}
E \Big( 
e^{- \frac{ \beta^2}{2} n \delta} \sum_{u \in N_{n \delta}} e^{- \beta |X^u_{n \delta}|} 
\Xi_B^{ \delta , \epsilon} (X^u_{n \delta}) \Big\vert \mathcal{F}_{n \delta} \Big) 
= e^{- \frac{ \beta^2}{2} n \delta} \sum_{u \in N_{n \delta}} e^{- \beta |X^u_{n \delta}|} 
\xi_B^{ \delta , \epsilon} (X^u_{n \delta}) \text{,}
\end{equation}
and
\begin{equation}
\label{new_eq10}
\lim_{n \to \infty} e^{- \frac{ \beta^2}{2} n \delta} \sum_{u \in N_{n \delta}} 
e^{- \beta |X^u_{n \delta}|} \xi_B^{ \delta , \epsilon} (X^u_{n \delta}) 
= \int \xi_B^{\delta, \epsilon}(x) \pi(\mathrm{d}x) M_{\infty} \text{.}
\end{equation}
The last equation follows from the SLLN along lattice times which we already proved. 
Also we should point out that if we further let $ \delta \to 0 $, $ \xi_B^{\delta, \epsilon}(x)$ 
will converge to $ \mathbf{1}_{B}(x)$ and \eqref{new_eq10} will converge to 
$ \pi(B) M_{ \infty}$. Our next step then is to show that
\begin{equation}
\label{new_eq11}
\lim_{n \to \infty} \Big\vert 
e^{- \frac{ \beta^2}{2} n \delta} \sum_{u \in N_{n \delta}} e^{- \beta |X^u_{n \delta}|} 
\Xi_B^{ \delta , \epsilon} (X^u_{n \delta}) - 
e^{- \frac{ \beta^2}{2} n \delta} \sum_{u \in N_{n \delta}} e^{- \beta |X^u_{n \delta}|} 
\xi_B^{ \delta , \epsilon} (X^u_{n \delta})
\Big\vert = 0 \text{.}
\end{equation}
In view of \eqref{new_eq8} and \eqref{new_eq9} we prove this using the method of Part I. 
That is, we exploit the Borel-Cantelli Lemma and in order to do that we need to show 
that for some $p \in (1,2) $
\begin{align*}
\sum_{n = 1}^{ \infty} E \bigg( \Big\vert 
e^{- \frac{ \beta^2}{2} n \delta} \sum_{u \in N_{n \delta}} e^{- \beta |X^u_{n \delta}|} 
\Xi_B^{ \delta , \epsilon} (X^u_{n \delta}) 
- E \Big( e^{- \frac{ \beta^2}{2} n \delta} \sum_{u \in N_{n \delta}} e^{- \beta |X^u_{n \delta}|} 
\Xi_B^{ \delta , \epsilon} (X^u_{n \delta}) \Big\vert \mathcal{F}_{n \delta} \Big) 
\Big\vert^p \bigg) < \infty \text{.}
\end{align*}
A similar argument to the one used in Part I (see \eqref{new_eq4} gives us that
\begin{align*}
&\sum_{n = 1}^{ \infty} E \bigg( \Big\vert e^{- \frac{ \beta^2}{2} n \delta} \sum_{u \in N_{n \delta}} e^{- \beta |X^u_{n \delta}|} 
\Xi_B^{ \delta , \epsilon} (X^u_{n \delta}) 
- E \Big( e^{- \frac{ \beta^2}{2} n \delta} \sum_{u \in N_{n \delta}} e^{- \beta |X^u_{n \delta}|} 
\Xi_B^{ \delta , \epsilon} (X^u_{n \delta}) \Big\vert \mathcal{F}_{n \delta} \Big) 
\Big\vert^p \bigg) \\
\leq &\sum_{n = 1}^{\infty} 2^{2p} e^{- p \frac{\beta^2}{2} n \delta} E \Big( \sum_{u \in N_{n \delta}} 
e^{- \beta p |X^u_{n \delta}|} \xi_B^{ \delta , \epsilon} (X^u_{n \delta}) \Big) \text{,}
\end{align*}
where $ \Xi_B^{ \delta , \epsilon} (X^u_{n \delta})$ is an indicator function 
and therefore raising it to the power $p$ leaves it unchanged. 
Using once again the Many-to-One Lemma and the usual change of measure we get 
\begin{align*}
&\sum_{n = 1}^{\infty} 2^{2p} e^{- p \frac{\beta^2}{2} n \delta} E \Big( \sum_{u \in N_{n \delta}} 
e^{- \beta p |X^u_{n \delta}|} \xi_B^{ \delta , \epsilon} (X^u_{n \delta}) \Big) \\
\leq &\sum_{n = 1}^{\infty} 2^{2p} e^{- p \frac{\beta^2}{2} n \delta} E \Big( 
\sum_{u \in N_{n \delta}} e^{- \beta p |X^u_{n \delta}|} \Big) \\
= &\sum_{n = 1}^{\infty} 2^{2p} e^{- (p - 1) \frac{\beta^2}{2} n \delta} E^{\tilde{Q}_{\beta}} \Big( 
\sum_{u \in N_{n \delta}} e^{- \beta (p - 1) |X^u_{n \delta}|} \Big) < \infty \text{.}
\end{align*}
Thus we have proved \eqref{new_eq11}, which together with \eqref{new_eq10} implies that
\begin{align*}
\liminf_{n \to \infty} e^{- \frac{ \beta^2}{2} n \delta} \sum_{u \in N_{n \delta}} 
e^{- \beta |X^u_{n \delta}|} \Xi_B^{ \delta , \epsilon} (X^u_{n \delta}) 
&= \liminf_{n \to \infty} e^{- \frac{ \beta^2}{2} n \delta} \sum_{u \in N_{n \delta}} 
e^{- \beta |X^u_{n \delta}|} \xi_B^{ \delta , \epsilon} (X^u_{n \delta}) \\
&= \int \xi_B^{\delta, \epsilon}(x) \pi( \mathrm{d}x) M_{ \infty} \text{.}
\end{align*}
Putting this into \eqref{new_eq7} and letting $n = \lfloor \frac{t}{\delta} \rfloor$ gives us
\[
\liminf_{t \to \infty} U_t \geq \frac{e^{- \frac{{\beta}^2}{2} \delta}}{1 + \epsilon} 
\int \xi_B^{\delta, \epsilon}(x) \pi( \mathrm{d}x) M_{ \infty} \text{.}
\]
Letting $ \delta, \epsilon \searrow 0$ we get $ \liminf_{t \to \infty} U_t \geq \pi(B) M_{ \infty}$. 
Since the same result also holds for $B^c$ we can easily see that
$ \limsup_{t \to \infty} U_t \leq \pi(B) M_{ \infty}$. Thus
\[
\lim_{t \to \infty} U_t = \pi(B) M_{ \infty} \text{.}
\]
Then the same argument as at the end of Part II of the proof extends the result for functions 
of the form $ \mathbf{1}_B(x) e^{- \beta |x|}$ to all bounded measurable functions. 
\end{proof}
\bibliographystyle{acm}

\def\cprime{$'$}

\end{document}